\documentclass[a4paper,12pt]{article}

\usepackage[a4paper]{geometry}

\usepackage[OT2,T1]{fontenc}

\usepackage{lipsum} 

\usepackage[english,francais]{babel}
\DeclareSymbolFont{cyrletters}{OT2}{wncyr}{m}{n}
\DeclareMathSymbol{\Sha}{\mathalpha}{cyrletters}{"58}

\usepackage{amssymb,amsmath,amsthm,amscd}
\usepackage{mathrsfs}

\usepackage{subfig}

\usepackage{tabularx} 
\usepackage{calc} 
\usepackage{graphicx}
\usepackage{hyperref}

\usepackage[nottoc]{tocbibind}

\usepackage{makeidx}
\makeindex

\usepackage{longtable}
\usepackage{enumerate}

\usepackage{amsthm}
\usepackage{amssymb}
\usepackage{amsmath}
\usepackage{array}
\usepackage{mathrsfs}
\usepackage[utf8]{inputenc}
\usepackage[T1]{fontenc}
\usepackage{version}
\usepackage{lscape}
\usepackage[all]{xy}
\usepackage{todonotes}
\usepackage{graphicx}
\usepackage{multirow}

\theoremstyle{definition}

\newtheorem{ex}{Example}
\theoremstyle{plain}

\newtheorem{thm}{Theorem}[section]
\newtheorem{lm}[thm]{Lemma}
\newtheorem{propo}[thm]{Proposition}
\newtheorem{coro}[thm]{Corollary}
\newtheorem{conjecture}[thm]{Conjecture}
\theoremstyle{remark}
\newtheorem{rem}{Remark}

\def\C{\mathbb{C}}
\def\Z{\mathbb{Z}}
\def\N{\mathbb{N}}
\def\Q{\mathbb{Q}}

\def\E{\mathscr{E}}
\def\P{\mathbb{P}}

\def\pgcd{\mathrm{pgcd}}

\title{On the density of rational points on rational elliptic surfaces}
\author{Julie DESJARDINS}
\date{}

\begin{document}
\begin{otherlanguage}{english}

\maketitle

\abstract{
Let $\E\rightarrow\P^1_\Q$ be a non-trivial rational elliptic surface over $\Q$ with base $\P^1_\Q$ (with a section). We conjecture that any non-trivial elliptic surface has a Zariski-dense set of $\Q$-rational points. In this paper we work on solving the conjecture in case $\E$ is rational by means of geometric and analytic methods. First, we show that for $\E$ rational, the set $\E(\Q)$ is Zariski-dense when $\E$ is isotrivial with non-zero $j$-invariant and when $\E$ is non-isotrivial with a fiber of type $II^*$, $III^*$, $IV^*$ or $I^*_m$ ($m\geq0$). We also use the parity conjecture to prove analytically the density on a certain family of isotrivial rational elliptic surfaces with $j=0$, and specify cases for which neither of our methods leads to the proof of our conjecture.}


\section{Introduction}

Let $\E$ be an elliptic surface over $\P^1_{\Q}$, i.e. a projective algebraic surface $\E$ defined over $\Q$ endowed with a morphism $\pi:\E\rightarrow \P^1_\Q$ such that, for all $t\in\P^1_\Q$ except a finite number, the fiber $\E_t:=\pi^{-1}(t)$ is a smooth projective curve of genus 1. Moreover, we suppose that there exists a section $\sigma:\P^1_\Q\rightarrow\E$ for $\pi$.

Such an elliptic surface can be written as the set of solutions in $\P_\Q^2\times\P^1_\Q$ of a Weierstrass equation \begin{equation}\label{ellipticsurfaceeqn}\E:y^2z=x^3+A(T)xz^2+B(T)z^3,\end{equation} with $A(T),B(T)\in\Z[T]$. We call \emph{generic fiber} of $\E$ the elliptic curve over $\Q(T)$, denoted by $\E_T$, whose model is given by the equation (\ref{ellipticsurfaceeqn}). We classify the elliptic surfaces according to their $j$-invariant function:
\begin{enumerate}
\item $\E$ is non-isotrivial if $j(T)$ is non-constant,
\item $\E$ is isotrivial otherwise, and admits a Weierstrass equation of the form 
\begin{enumerate}
\item $y^2=x^3+af(T)^2x+bf(t)^3$ where $a,b$ non-zero integers and $f\in\Q[T]$, (if $j(T)\in\Q\backslash\{0,1728\}$)
\item $y^2=x^3+A(T)$ where $A\in\Q[T]$ (if $j(T)=0$),
\item $y^2=x^3+B(T)x$ where $B\in\Q[T]$ (if $j(T)=1728$).
\end{enumerate}
\end{enumerate}

For almost every $t\in\P^1_\Q$, the fiber $\E_t=\pi^{-1}(t)$ is an elliptic curve over $\Q$ and the set $\E_t(\Q)$ admits a group structure. By Mordell-Weil's theorem, this group decomposes as the sum of a finitely generated free group (isomorphic to $\Z^r$) and a finite group (the torsion points). The integer $r=\mathrm{rk} (\E_t)$ is called the $\emph{Mordell-Weil rank}$ (or simply the \emph{rank}) of $\E_t$ over $\Q$.


We put forward the following conjecture, a variant of a conjecture of Mazur \cite[Conjecture 4]{Mazur}, where "real density" is replaced by "Zariski density".
This conjecture is already implicit in the literature, particularly in \cite{HCC}.

\begin{conjecture}\label{maconjecture}
Let $\E\rightarrow\P^1_\Q$ be a non-trivial elliptic surface over $\Q$. Then 
$\E(\Q)$ is Zariski-dense in $\E$. 
\end{conjecture}

\begin{rem}
In case $\E$ is a trivial surface, i.e. there exists a curve $E_0$ such that $\E\simeq E_0\times\P^1_\Q$ over $\Q$. One has $\E(\Q)=E_0(\Q)\times\P^1_\Q(\Q)$, and this can be dense or not dense depending on the number of points of $E_0(\Q)$.
\end{rem}

There is two approaches to solve this conjecture. We can either prove that the rank is non-zero by means of geometric arguments, or we can use the parity conjecture (a weaker version of the Birch and Swinnerton-Dyer conjecture) which links the root number of an elliptic curve $E$ to the parity of its rank :
$$W(E)=(-1)^{\mathrm{rk} E}.$$

We already have evidence (see \cite{Manduchi, Helfgott, Desjardins1}) that the rational points should be dense when the elliptic surface is non-isotrivial, because of the variation of the root number of the fibers. The articles mentioned above additionally use two conjectures of analytic number theory, the squarefree conjecture and Chowla's conjecture, which are known only for polynomials of low degree.

In this article, we work on proving the conjecture in case $\E$ is a rational elliptic surface, i.e. $\overline{\Q}$-birational to $\P^2$. This is the class of elliptic surfaces with the simplest geometry, and thus a good starting point. It is also interesting to observe that it leads to results on Del Pezzo surfaces, due to the relation between those classes of surfaces.
In certain cases, we can even prove the stronger property of \emph{unirationality}: a surface $S$ over a field $k$ is unirational if there is a dominant rational map $\P^2\dashrightarrow  \E$ over $k$. Note that if a projective surface $S$ is $\Q$-unirational, one has in particular that the set of its rational points $S(\Q)$ is Zariski-dense in $S$.

Finally, note that in section \ref{inconditionnelles}, we state on which rational elliptic surfaces the results of \cite{Helfgott} unconditionnally imply the variation of the root number.

We prove the following theorem :

\begin{thm}\label{thmprincipal}
Let $\E$ be a rational elliptic surface.
\begin{enumerate}
\item Suppose $\E$ is isotrivial with non-zero $j$-invariant.
\begin{enumerate}
\item Then the set of rational points $\E(\Q)$ is Zariski-dense in $\E$. 
\item Moreover, if the $j$-invariant is not $1728$, the surface $\E$ is $\Q$-unirational.
\end{enumerate}
\item Suppose $\E$ has a fiber of type $II^*$, $III^*$, $IV^*$ or $I^*_m$ ($m\geq0$). Then $\E$ is $\Q$-unirational.
\end{enumerate}
\end{thm}

When $\E$ is isotrivial with $j(T)=0$, this theorem says nothing on the Zariski-density of the rational points. Such surfaces are given by a Weierstrass equation of the form $y^2=x^3+g(T)$, where $g(T)\in\Z[T]$ is a polynomial of degree at most 6. According to V\'arilly-Alvarado \cite{VA}, the root number of the fibers of one of these surfaces always takes infinitely many negative values, except possibly when all $P_i$ irreducible factors of $g$ are such that
$$\mu_3\subseteq\Q[t]/P_i(t).$$
However, it may happen that the polynomial $g$ does not respect this condition, in particular $g(T)=AT^6+B$, where $A,B\in\Z$ are such that $\frac{3A}{B}$ is a rational square. 
Theorem \ref{conditions0}, gives precise conditions on the integers $A$ and $B$ for the surface $$\E_{A,B,C} : y^2=x^3+AT^6+B$$ to have a constant root number of the fibers always $+1$. This gives a description of surfaces for which our methods does not allow to prove Conjecture \ref{maconjecture}.

We also study in Section \ref{IRES1728root} the variation of the root number on a specific family of rational elliptic surfaces with $j(T)=1728$. It is interesting to note that in the case we study the proof of Theorem \ref{thmprincipal} gives a section of finite order. In
Theorem \ref{conditions1728}, give the conditions under which there can be a constant root number on the fibers of an elliptic surface $\mathscr{F}_{A,B,C}$ given by the following Weierstrass equation:
$$\mathscr{F}_{A,B,C}:y^2=x^3+C(A^2T^4+B^2)x\text{, }A,B,C\in\Z\text{ and }A,B\text{ coprime.}$$

\subsection{Previous results}

Various geometric arguments allow one to prove unconditionnally the density of rational points on rational elliptic surfaces (or its associated del Pezzo surface). 

Rohrlich \cite[Theorem 3]{Rohr} shows the Zariski-density on $$f(t)y^2=x^2+ax+b$$ where $f$ is a quadratic polynomial and $a,b$ are non-zero integer, under the additional assumption that there exists a fiber of positive rank. 

In \cite{SalgRational}, Salgado studies the problem of comparing the rank of the special fibers over a number field $k$ with that of the generic fiber over $k(\P^1)$. She proves for a large class of rational elliptic surfaces the existence of infinitely many fibers whose rank exceeds the generic rank of at least 2.

In \cite{Ulas} and \cite{Ulas2}, Ulas proves the density of rational points on certain families of isotrivial rational elliptic surfaces with $j$-invariant 0 and 1728 by constructing a multisection with infinitely many points on those families. Jabara generalizes one of Ulas' work in \cite[Theorems C and D]{Jabara} and proves the density when $B(t,1)$ is monic and the pair of coefficients $(A,B)$ is sufficiently general.
An article of Salgado and Van Luijk \cite{SVL} improves Ulas construction, and proves the Zariski-density of set of rational points of a del Pezzo surface of degree 1 satisfying certain conditions. For instance it suffices to suppose that the elliptic surface obtained by the blowup of the anticanonical point has a fiber of type $II$ over a certain $k$-rational point of $\P^1$. 

The approach of Bettin, David and Delaunay \cite{biasedfamily} is another way to find whether or not a rational elliptic surface has a section over $\Q$. They study specifically the elliptic surfaces given by a Weierstrass equation $y^2=x^3+a_2(T)x^2+a_4(T)x+a_6(T)$ where $a_2,a_4,a_6\in\mathbb{Z}[X]$ wih no place of multiplicative reduction except possibly at infinity. They find different classes of such families such that $\deg a_i(t)\leq2$ for $i=2,4,6$ and on each of them compute the generic rank. In particular, this proves the density of rational points on those of them with a non-zero generic rank.

Rohrlich pioneered the study of variations of root numbers on algebraic families of elliptic curves in \cite{Rohr}. Many authors followed: see, for example, \cite{Manduchi, GM, Rizz, HCC, Helfgott, VA}.
Some authors (notably \cite{CS},\cite{VA}) remarked that it can happen that the root number of the fibers might all be of the same value, when the elliptic surface considered is isotrivial, i.e. its modular invariant $j(E)$ has no $T$-dependence.

\subsection{Outline of the paper}
In Section \ref{RES}, we give a few reminders on rational elliptic surfaces and del Pezzo surfaces. 
In Section \ref{rootnumber}, we recall the definition and the properties of the root number.

In Section \ref{IRESgeneral}, we prove the unirationality of rational elliptic surfaces with $j$-invariant different from $0$ or $1728$ (the second point of Theorem \ref{thmprincipal}). 
In Section \ref{IRES1728}, we exhibit a section on a rational elliptic surface with $j$-invariant equal to $1728$ and from this deduce the density of its rational points. This section is not always of infinite order, but its existence completes the proof of the statement on isotrivial rational elliptic surfaces of Theorem \ref{thmprincipal}.

In Section \ref{IRES0root}, we find conditions on the coefficients of an rational elliptic surface with zero $j$-invariant give by the equation $y^2=x^3+AT^6+B$ (with $A,B\in\Z$) so that the root number of the fibers always takes the value $+1$. In Section \ref{IRES1728root}, we find conditions on the coefficients of some rational elliptic surfaces with $j$-invariant $1728$ given by the equation $y^2=x^3+xC(A^2T^4+B^2)$ (where $A,B,C\in\Z$) so that the root number of the fibers always takes the value $+1$.

We end the article in Section \ref{noniso} with the completion of the proof of Theorem \ref{thmprincipal}. We also give various small results on non-isotrivial elliptic surfaces. 

\subsection{Acknowledgements}

I thank my supervisor, M. Hindry, for numerous helpful conversations and suggestions and for his encouragement. I thank D. Rohrlich and J.-M. Couveignes for their careful reading of earlier versions of this work. I also thank the anonymous referee for good suggestions.

Most of the mathematics of this paper were done at Institut de Math\'ematiques de Jussieu - PRG. I thank the Institut Fourier of Universit\'e Grenoble Alpes and the Max Planck Institute in Bonn for providing good working environment.

\section{Rational elliptic surfaces}\label{RES}

Let $\E\rightarrow\P^1_\Q$ be an elliptic surface over $\Q$ given by a minimal Weierstrass equation
\begin{displaymath}
\E:y^2=x^3+A(T)x+B(T),
\end{displaymath}
where $A,B\in\Z[T]$. The discriminant is the homogeneous polynomial defined as $$\Delta_\E(X,Y)=Y^{12k-\deg \Delta}\Delta(X/Y)$$ where $\Delta(T)=4A(T)^3+27B(T)^2$, and $k$ is the smallest integer such that $12k\geq\deg\Delta(T)$. Note that one has thus $\deg \Delta_\E=12k$.

\begin{propo} (Criteria of rationality \cite{Miranda})

An elliptic surface is rational, if and only if
$$0<\max\{3\deg A,2\deg B\}\leq12$$
\end{propo}

We observe thus that the discriminant $\Delta_\E$ actually gives the following classification of elliptic surfaces:
$$\deg \Delta_\E= \begin{cases}
0 & \E\text{ is trivial}\\
12 & \E\text{ is rational}\\
24 & \E\text{ is a $K3$-surface}\\
... & \text{...}
\end{cases}
$$

Rational elliptic surfaces are the (non-trivial) elliptic surfaces with discriminant of lowest degree, and studying the density on them is a first step towards the resolution of Conjecture \ref{maconjecture}.

\subsection{Minimal model of a rational elliptic surface}\label{MM}
 The following theorem due to Iskovskikh links rational elliptic surfaces to Del Pezzo surfaces.

\begin{thm}\label{Isk} \cite[Thm. 1]{Isk}

Let $\E$ be a rational elliptic surface defined over $\Q$.

Then, it has a minimal model $X\slash\Q$ that is :
\begin{enumerate}
\item either a conic bundle of degree $\geq1$,
\item or a Del Pezzo surface.
\end{enumerate}
\end{thm}

A \emph{del Pezzo surface} $X$ is a non-singular projective algebraic surface whose anticanonical divisor is ample. 
Its \emph{degree} is the integer $d\in\{1,\dots,9\}$ corresponding to the self-intersection number $(K_X,K_X)$ of the canonical divisor of $X$. 


When $X$ is a conic bundle, the work of Kollar and Mella \cite{KollMell} guarantees that the surface is $\Q$-unirational, i.e. it is dominated by the projective plane $\P^2\dashrightarrow X$. In particular, the set of rational points is dense. 

Suppose that $X$ is a del Pezzo surface of degree $d$.
When $d\geq3$, one knows by the work of Segre and Manin \cite{Manin} that the existence of one rational point on $X$ implies that the surface is $\Q$-unirational.
When $d=2$,  Salgado, Testa and Várilly-Alvarado \cite{STVA}, based on a work of Manin \cite[Thm 29.4]{Manin}, showed that if $X$ contains a rational point that does not lie on an exceptional curve nor a certain quartic, then $X(\Q)$ is Zariski-dense.
If $d=1$, the surface $X$ has automatically a rational point: the base point of the anticanonical system. However, the results concerning density of rational points are still partial 
(for instance \cite{SVL} and \cite{VA}).

\subsection{Del Pezzo surfaces of degree one}\label{DP1}
If we blow up the anticanonical point on $X$, a del Pezzo surface of degree 1, we obtain a rational elliptic surface $\E$ such that the image of the neutral section is the exceptional divisor.
Thus, the rational points of $X$ are dense if and only if the rational points of $\E$ are dense.

By studying the singular points on rational elliptic surfaces, we obtain the following lemma:

\begin{lm} \label{thmDP1}
Let $\E$ be a minimal rational elliptic surface. We denote by $X$ the surface obtained from $\E$ by contracting its section at infinity.
Then $X$ is a del Pezzo surface of degree 1 if and only if the only singular fibers of $\E$ have type $II$ or $I_1$.

\end{lm}

\begin{proof}
A del Pezzo surface is smooth by definition. Therefore, the blow-up of its base point also gives a smooth elliptic surface, meaning that the only singular fibers are irreducible (in other words, those fibers have type $I_1$ or $II$). 
\end{proof}

\subsection{Isotrivial rational elliptic surfaces}\label{IRES}

An isotrivial rational elliptic surface takes one of the following forms:
\begin{enumerate}
\item $y^2=x^3+aH(u,v)^2x+bH(u,v)^3$ where $a,b\in\Q^*$ are such that $4a^3+27b^2\not=0$ (if $j\in\Q\backslash\{0,1728\}$);
\item $y^2=x^3+A(u,v)x$ (if $j=0$);
\item $y^2=x^3+B(u,v)$ (if $j=1728$),
\end{enumerate}
for polynomial $A,B,H\in\Z[u,v]$ such that $\deg H\leq2$, $\deg A\leq4$ and $\deg B\leq6$.
To avoid the case where the surface is trivial, we suppose also that $H$ is not a square, $A$ is not a $4$th-power and $B$ is not a $6$th-power.

In each cases, the singular fibers have the following configuration:

\begin{enumerate}
\item Every singular fiber has type $I_0^*$.
\item The singular fibers have either type $I_0^*$, $III$ or $III^*$.
\item The singular fibers have either type $I_0^*$, $II$, $II^*$, $IV$ or $IV^*$.
\end{enumerate}

The only case where an isotrivial rational elliptic surface 
has a del Pezzo surface of degree 1 as a minimal model is the third one, when moreover the polynomial $B$ is squarefree and has degree $\geq5$. 

\section{Root number}\label{rootnumber}

\subsection{Definition and motivation}
The root number of an elliptic curve $E$ is expressed as the product of the local factors $$W(E)=\prod_{p\leq\infty}{W_p(E)},$$
where $p$ runs through the finite and infinite places of $\Q$, $W_p(E)\in\{\pm1\}$ and $W_p(E)=+1$ for all $p$ except a finite number of them. The \emph{local root number of $E$ in $p$}, $W_p(E)$, is defined in terms of the epsilon factors of the Weil-Deligne representations of $\Q_p$ (see \cite{Del} and \cite{Tat}). Rohrlich \cite{Rohr} gives an explicit formula for the local root numbers in terms of the reduction of the elliptic curve $E$ at a prime $p\not=2,3$ and at $p=2,3$ in case $E$ is semi-stable. Halberstadt \cite{Halb} gives tables (completed by Rizzo \cite{Rizz}) for the local root number at $p=2,3$ according to the coefficients of $E$. Moreover we always have $W_\infty(E)=-1$.

The root number is hypothetically equal to the sign $W(E)\in\{\pm1\}$ of the conjectural functional equation of $L(E,s)$ the $L$-function of $E$ :
\begin{displaymath}
\mathscr{N}_E^{(2-s)/2}(2\pi)^{s-2}\Gamma(2-s)L(E,2-s)=W(E)\mathscr{N}_E^{s/2}(2\pi)^{-s}\Gamma(s)L(E,s).
\end{displaymath}

When we work on elliptic curves over $\Q$, such a functional equation always exists (by Wiles' work \cite{Wiles} and its generalisation by Breuil, Conrad, Diamond and Taylor \cite{bcdt-fermat}) and the values of the root number and the sign of the functional equation are indeed the same.

The Birch and Swinnerton-Dyer conjecture implies that the root number is related to the rank of the elliptic curve:

\begin{conjecture}[Parity Conjecture]
For all elliptic curve $E$ over $\Q$, we have
$$W(E)=(-1)^{\mathrm{rank} \ E(\Q)}.$$
\end{conjecture}

As a consequence of this equality, it suffices that $W(E)=-1$ for the rank of $E(\Q)$ not to be zero and in particular for $E(\Q)$ to be infinite.

Let $\E$ be a rational elliptic surface over $\P^1_\Q$. The elliptic surface can be seen as a family of elliptic curves, and admits a Weierstrass equation of the form

$$\mathscr{E}:y^2=x^3+A(T)x+B(T),$$ with $A(T), B(T)\in\Z[T]$ have respectively degree less than or equal to 4 and 6. 

We denote by $\Delta(T)=4A(T)^3-27B(T)^2$ the discriminant and the corresponding homogenous polynomial $$\Delta_\E(X,Y)=Y^{12-\deg \Delta}\Delta(X/Y).$$ Let also $M_\E(X,Y)$ the product of the polynomials associated to the places of multiplicative reduction, that is to say, polynomials dividing $\Delta_\E$, but not $Y^{4-\deg A}A(X/Y)$.

We consider the sets $W_+$ and $W_-$ given by $$W_\pm(\E)=\{t\in\Q: \E_t\text{ is an elliptic curve and }W(\E_t)=\pm1\}.$$

As a consequence of the parity conjecture, if $\#W_-(\E)=\infty$, then there exist infinitely many fibers of $\E$ that are non singular elliptic curves with positive rank, and this guarantees the density of the rational points on $\E$.

When the surface is isotrivial, it can happen that one of the set $W_-$ or $W_+$ is finite or empty. 
In \cite{CS}, Cassels and Schinzel find a family of elliptic curves, such that $j(T)=1728$, on which the sign of the fibers always takes the value $-1$: $$\E_T:y^2=x^3-(7+7T^4)^2x.$$

Varilly-Alvarado gives more examples of elliptic surfaces with constant root number in \cite{VA}, among them the following elliptic surface with $j=0$, given by the Weierstrass equation
$$y^2=x^3+6(27T^6+1),$$
whose fibers always have a root number of value $+1$.

\subsection{Local root number at 2 and 3 of $y^2=x^3+\alpha x$ and $y^2=x^3+\alpha$}

We give here some formulas for the local root number at 2 and 3 of the elliptic curves $y^2=x^3+\alpha x$ and $y^2=x^3+\alpha$ for $\alpha\in\Q$. 

\begin{lm}\cite[Lemme 4.7]{VA}\label{VAx}

Let $t$ be a non-zero integer and let be the elliptic curve $E_t:y^2=x^3+tx$. We denote by $W_2(t)$ and $W_3(t)$ its local root numbers at 2 and 3. Put $t_2$ and $t_3$ the integers such that $t=2^{v_2(t)}t_2=3^{v_3(t)}t_3$. Then
\begin{displaymath}
W_2(t)= 
\begin{cases}
 -1 & \text{if }v_2(t)\equiv1\text{ or } 3\mod4 \text{ and }t_2\equiv1\text{ or }3\mod8; \\
  & \text{or if }v_2(t)\equiv0\mod4 \text{ and } t_2\equiv1,5,9,11,13\text{ or }15\mod16;\\
   & \text{or if }v_2(t)\equiv2\mod4 \text{ and }t_2\equiv1,3,5,7,11,\text{ or }15\mod16;\\
 +1 & \text{otherwise.}
\end{cases}
\end{displaymath}

\begin{displaymath}
W_3(t)=
\begin{cases}
 -1 & \text{if }v_3(t)\equiv2\mod4; \\
 +1 & \text{otherwise.}
\end{cases}
\end{displaymath}
\end{lm}

\begin{lm}\label{VA} \cite[Lemme 4.1]{VA}

Let $t$ be a non-zero integer and the elliptic curve $E_t:y^2=x^3+t$. We denote by $W_2(t)$ and $W_3(t)$ its local root numbers at 2 and 3. Put $t_2$ and $t_3$ the integers such that $t=2^{v_2(t)}t_2=3^{v_3(t)}t_3$. Then
\begin{displaymath}
W_2(t)= 
\begin{cases}
 -1 & \text{if }v_2(t)\equiv0\text{ or } 2\mod6; \\
  & \text{or if }v_2(t)\equiv1,3,4\text{ or }5\mod6\text{ and }t_2\equiv3\mod4;\\
 +1, & \text{otherwise.}
\end{cases}
\end{displaymath}

\begin{displaymath}
W_3(t)=
\begin{cases}
 -1 & \text{if }v_3(t)\equiv1\text{ or } 2\mod6\text{ and }t_3\equiv1\mod3; \\
  & \text{or if }v_3(t)\equiv4\text{ or }5\mod6\text{ and }t_3\equiv2\mod3;\\
  & \text{or if }v_3(t)\equiv0\mod6\text{ and }t_3\equiv5\text{ or }7\mod9;\\
  & \text{or if }v_3(t)\equiv3\mod6\text{ and }t_3\equiv2\text{ or }4\mod9;\\
 +1, & \text{otherwise.}
\end{cases}
\end{displaymath}
\end{lm}

\section{Isotrivial rational elliptic surfaces with $j(T)\not=0,1728$}\label{IRESgeneral}

\subsection{A theorem of Kollar and Mella}

\begin{thm}\label{KollarMella} \cite[Thm. 1]{KollMell}
Let $K$ be any field of characteristic $\not=2$ and $a_0(t),\dots,a_3(t)\in K[t]$ polynomials of degree 2 giving a nontrivial family of elliptic curves. Then the surface
\begin{displaymath}
S:y^2=a_3(t)x^3+a_2(t)x^2+a_1(t)x+a_0(t)\subset\mathbb{A}^3_{xyt}
\end{displaymath}
is unirational over $K$. 
\end{thm}

In a first version of the article of  Koll\'ar-Mella \cite{KollMell}, Theorem \ref{KollarMella} excluded the isotrivial case. The author wanted to complete this result, and obtained Theorem \ref{thmuni}. However, it had been completed by Koll\'ar and Mella themselves by the time she submitted her ph.D thesis. Their technique is different from the one in this article.

\subsection{A non-isotrivial elliptic fibration}

\begin{thm}\label{thmuni}
Let $\E$ be a isotrivial rational elliptic surface given by the equation
\begin{displaymath}
\E: Y^2=X^3+aH(T)^2X+bH(T)^3,
\end{displaymath}
where $a,b\in\Z\backslash\{0\}$ and $H(T)$ is a degree $\leq2$ polynomial that is not a square.
Then the surface $\E$ is $\Q$-unirational. In particular, its rational points are dense for Zariski topology.
\end{thm}

\begin{rem}
This result is proven by Rohrlich \cite[Theorem 3]{Rohr} under the \emph{a priori} restrictive assumption that there exists a fiber of positive rank. This assumption is removed here.
\end{rem}

\begin{proof}
Observe that the surface $\E$ is endowed with many fibrations.
\begin{center}
\ \xymatrix{
& \E : H(T)Y^2=X^3+aX+b \ar@{|->}[ld]_{\varphi_1} \ar@{|->}[d]_{\varphi_2} \ar@{|->}[rd]^{\varphi_3} \\
x 
&y 
&t}
\end{center}

The last two, $\varphi_2$ and $\varphi_3$, are elliptic fibration (with section). Even if the fibration defined by $\varphi_3$ is isotrivial, the one defined by $\varphi_2$ is not.
Indeed, if we write $H(T)=\alpha_2T^2+\alpha_1T+\alpha_0$ for the appropriate coefficients $\alpha_i$, the fibration $\varphi_2$ has the fiber
\begin{displaymath}
\E_y:=\alpha_2y^2T^2+\alpha_1y^2T=X^3+aX+b-\alpha_0y^2
\end{displaymath}
which can after a change of variables (first $T'=\alpha_2^2yT$ and $x=\alpha_2X$, then $t=T'+\frac{\alpha_1\alpha_2y}{2}$) be written
\begin{displaymath}
\E_y:t^2=x^3-27c_4(y)x-54c_6(y)
\end{displaymath}
where $c_4(y)=\alpha_2^2 a$ and $c_6(y)=(\alpha_2^3\alpha_0+\frac{\alpha_1^2\alpha_2^2}{4})y^2+\alpha_2^3 b$. 
By computing the $j$-invariant, one sees that this curve is not isotrivial, except in the case where $a=0$ ($c_4(y)$ is zero) and $\alpha_0=\alpha_2^3\alpha_0+\frac{1}{4}\alpha_1^2\alpha_2^2$, in other words when $H$ is the square of a linear polynomial (in that case, $\E$ is trivial). These cases are excluded by our hypotheses.
Hence, we can apply Theorem \ref{KollarMella}. This proves the unirationality of $\E$ endowed with the elliptic fibration $\varphi_2$.
\end{proof}

\begin{rem}
Another way to prove Theorem \ref{thmuni} would have been the use the work of Colliot-Thélène \cite{CT}. The second theorem of this article shows that for $X$, a conic bundle of degree 4, the Brauer-Manin obstruction to the Hasse principle is the only obstruction.
To deduce from this Theorem \ref{thmuni}, one would have to check that the Brauer group of the surfaces that we consider (whose equation is $h(t)y^2=x^3-ax$ where $\deg h=2$) is the Brauer-group of $\Q$.
\end{rem}

\section{Isotrivial rational elliptic surfaces with $j(T)=1728$}\label{IRES1728}

\subsection{A section of infinite order}\label{constructionsection}

We study now the isotrivial rational elliptic surfaces of the form $y^2=x^3+xA(T)$ where $A\in\Z[T]$ is such that $\deg A\leq4$.
The density of rational point is proven in the case where $\deg A\leq3$ by Ulas in \cite{Ulas2}. For this reason, we concentrate on surfaces such that $\deg A=4$. Let $a_4,a_3,a_2,a_1,a_0\in\Z$ be the coefficients such that $$A(T)=a_4T^4+a_3T^3+a_2T^2+a_1T+a_0.$$
First observe that we have $F(T)=a_4\left(\left(T^2+g_1T+g_0\right)^2+h_1T+h_0\right)$, where $$g_0=\frac{4a_2a_4-a_3^2}{8a_4^2}\text{, }g_1=\frac{a_3}{2a_4}\text{, }h_0=\frac{2^6a_4^3(a_0+x^2)-(4a_2a_4-a_3^2)^2}{2^6a_4^4}$$ and $$h_1=\frac{2^3a_1a_4^2-a_3(4a_2a_4-a_3^2)}{2^3a_4^3}.$$
We make the change of variables $T'=T-g_1/2$.
Thus we can write $$A(T)=a_4(T'^4+(\frac{-g_1^2}{2}+2g_0)T^2+(\frac{-g_1^3}{2}+2g_1g_0)T+(\frac{-g_1^4}{2^4}+g_0^2+h_0)).$$
Replacing $T^2$ and $T$ by their expressions in terms of $T'$,
we obtain the following equation: $$y^2=x^3 + a_4x(T'^4+A_2T'^2+A_1T'+A_0),$$
where $$A_2=(g_0-\frac{g_1^2}{2})\text{, }\quad
A_1=(\frac{g_1^3}{2}+a_2g_1)\text{, and }\quad
A_0=(-\frac{g_1^4}{2^4}+g_0^2+h_0).$$
Hence, one can assume that $a_3=0$ (or else we do the change of variable previously explained).
The surface $\E$ has the following fibrations.
\begin{center}
\ \xymatrix{
& \E : Y^2=X^3+A(T)X \ar@{|->}[ld]_{\varphi_1} \ar@{|->}[d]_{\varphi_2} \ar@{|->}[rd]^{\varphi_3} \\
x 
&y 
&t}
\end{center}

The initial fibration is $\varphi_3$. The fiber $\varphi_1: (x,y,t)\mapsto x$ is a genus 1 curve, and this fibration is a priori without section\footnote{To ensure the existence of a section to $\varphi_1:(x,y,t)\mapsto x$, it would be necessary to check that every $\varphi_1^{-1}(x)$ admits a rational point.}.

The equation of the fiber at $x$ can be written as:
\begin{displaymath}
C_x: y^2=a_4xt^4+a_2xt^2+a_1xt+(a_0x+x^3).
\end{displaymath}

It is a genus 1 curve with two points at infinity, denoted by $\infty_+$ and $\infty_-$, which are rational if and only if $x\in a_4\Q^{*2}$.

\begin{propo}\label{ptnontorsion}
Let $P_x=cl((\infty_+)-(\infty_-))\in C_x(\Q)$ for $x\in a_4\Q^{*2}$.

Then 
\begin{itemize}
\item if $a_1=0,$ the point $P_x$ has order 2,
\item if $a_1\not=0$, $P_x$ has infinite order (except for finitely many $x$).
\end{itemize}
\end{propo}

\begin{proof}

Explicitely, putting $u=\frac{1}{t}$ and $v=\frac{y}{t^2}$, one has in coordinate $(u,v)$ :
$$\infty_+=(0,b)\text{, and }\infty_-=(0,-b).$$
Suppose that $b^2=a_4x$ for a certain rational number $b$. We write 
\begin{displaymath}
C_x:y^2=b^2t^4+a_2xt^2+a_1xt+a_0x+x^3.
\end{displaymath}
A well-chosen change of variables\footnote{We use here a very classical method, explained in particular in a book of Cassels' \cite{Cassels}. An interested reader can also find the details of the change of variables in the author's phD thesis \cite[Section 1.1.3]{Desjardinsthese}. Observe that the coefficients $g_0,h_0,h_1$ in the general Weierstrass equation (\ref{weierstrassgenerale}) correspond to the quantities previously defined in this Section.} gives the following general Weierstrass equation for $C_x$.

\begin{equation}\label{weierstrassgenerale}
C_x : S^2+\frac{h_1}{4}S=R^3-g_0R^2-\frac{h_0}{4}R,
\end{equation}
where $$g_0=\frac{a_2}{2a_4}\text{, }\quad h_0=\frac{2^2a_4(a_0+x^2)-a_2^2}{2^2a_4^2}\quad\text{ and }h_1=\frac{a_1}{a_4}.$$
The two points at infinity are send to the two obvious points of (\ref{weierstrassgenerale}). We have:
$$R(\infty_+)=\infty \quad\text{ and }S(\infty_+)=\infty .$$
$$R(\infty_-)=0\quad\text{ and }S(\infty_-)=-\frac{h_1}{4}$$

We put in a natural way the point obtain from $\infty_+$ as marked point of the curve $C_x$, that is as the identity element of the group law of the set of rational points. With this configuration one has $(0,-\frac{h_1}{4})=[-1](0,0)$.
We deduce of this that if $h_1=\frac{a_1}{a_4}=0$, then the point $(0,-\frac{h_1}{4})$ has order 2.

In the case where $h_1\not=0$, let us find the order of $Q=(0,0)$. Its order will be the same as the one obtained from $\infty_-$.
We use a result proven simultaneously by Lutz and Nagell which can be found \cite[p.240]{Silv1}: if
 $E/\Q$ is an elliptic curve of Weierstrass equation $y^2=x^3+Ax+B$, $A,B\in\Z$ and that
 $P\in E(\Q)$ is a torsion point different for the point at infinity, then the following properties hold:
\begin{enumerate}
\item $x(P),y(P)\in\Z$.
\item We have either $[2]P=O$ or $x([2]P)\in\Z$. 
\end{enumerate}

To use this fact, we need to consider a curve with integer coefficients (denote these coefficients by $A_i$). As the coefficients of $C_x$ might not be integers, we will choose a certain integer $\alpha$ for which the twist of the curve has integer coefficients.
Let $u$ and $v$ be the coprime integers such that $x=\frac{u}{v}$. If we put $\alpha=2\cdot a_4v$, the coefficients of the curve are integers
$$C_x':S^2+\alpha^3\frac{h_1}{4}S=R^3-\alpha^2g_0R^2-\alpha^4\frac{h_0}{4}R.$$
(In fact, it is sufficient to put $\alpha=2a_4w$ where $w$ is an integer such that $v^2\mid w$.)
We now show that if $h_1\not=0$, the point $Q$ is not 2-torsion for infinitely many values of $x$.  
We first find the condition for $R([2]Q)$ to be an integer. 
We have:
$$R([2]Q)=\Big(\frac{4\alpha^4h_0}{4\alpha^3 h_1}\Big)^2+4\alpha^2g_0.$$
For this coordinate to be an integer, we need $\alpha^3h_1$ to divide $\alpha^4h_0$. 
Recall that  $x=\frac{u}{v}$ where $u,v\in\Z$ are coprime. We have 
$$\alpha^4h_0=A\left(\frac{u}{v}\right)^2+B,$$
where $A=2^4a_4^3v^4$ and $B=2^4a_4^3a_0u^4-2^2a_4^2a_2^2v^4.$
As for the quantity $\alpha^3 h_1$, it is an integer of value $$\alpha^3 h_1=2^3a_4^2a_1v^3.$$ 

If $\alpha^3 h_1\mid \alpha^4 h_0$ for every $x\in\Q^{2*}a_4$, then $\alpha^3 h_1$ divides $B$ (we obtain this taking for instance $x=(\alpha^3 h_1)^2 a_4$ ). Therefore, $\alpha^3 h_1$ divides $Ax^2$ for any choice of $x$. Choose $v$ prime to $2a_4$. In this case, we have a contradiction since $Ax^2=2^3a_4^2v^2(2a_4u)$ should be divisible by $2^3a_4^2a_1v^3$, but $v$ is assumed to be prime to $2a_4$ and to $u$. This contradiction shows that for every $x\in\Q^{2*}a_4$ whose denominator is prime to $2a_4$, the point $Q$ is of infinite order on the curve $C_x$.

We conclude the proof by using Silverman's specialization theorem (see \cite{Silvermanspec} and \cite[Theorem 11.4, Chapter III]{Silv2}).
A priori, the fibration
\begin{align*}
  \varphi_2 \colon \E &\to \P^1_\Q\\
  (x,y,t) &\mapsto x.
\end{align*}
is not an elliptic surface over $\Q$.
However, let us consider the application
\begin{align*}
  \phi \colon \P^1_\Q &\to \P^1_\Q\\
  z &\mapsto x=a_4z^2.
\end{align*}
and the fibered product $\E'$ of $\E$ with respect to the fibration. By the previous argument, $\E'$ admits two sections $\infty_+$ and $\infty_-$. It is thus an elliptic surface over $\Q$. Let us choose as the canonic section $\infty_+$.

If there exists a linear change of variable such that $A=A_4T'^4+A_2T^2+A_0$, then $\infty_-$ is a torsion point on every fiber at $x=az^2$ of $\E$. Therefore, the section $\infty_-(z)$ is torsion for every $z\in\P^1_\Q$ (except finitely many, i.e. those defining a singular fiber).

If there is no such change of variable, then the point $\infty_-$ has infinite order for infinitely many fibers of $\E$. Therefore, Silverman's specialization theorem guarantees that $\infty_-(z)$ has infinite order on every fiber of $\E'$ except for finitely many $z$.
\end{proof}

We directly deduce from this proposition the following theorem :

\begin{thm}\label{4thmA}
Let $\E$ be a rational elliptic surface given by the equation
\begin{displaymath}
\E:Y^2=X^3+A(T)X,
\end{displaymath}
where $A(T)$ is a polynomial of degree 4 with integer coefficients. 

Suppose there exists no linear change of variable $T\rightarrow T'+b$ such that $A$ is of the form $$A(T')=A_4T'^4+A_2T'^2+A_0,$$ where $A_4,A_2,A_0\in\Z$.

Then the rational points of $\E$ are Zariski-dense.
\end{thm}

\begin{rem}
The surfaces which are not treated by this theorem are of the form: $$y^2=x^3+x(a_4T^4+4ba_4T^3+(6b^2a_4+a_2)T^2+(4b^3+2ba_2)T+a_4b^4+a_2b^2+a_0)$$ for a certain $b\in\Q$ and $a_4,a_2,a_0\in\Z$ such that $\sqrt{a_2^2-4a_4a_0}\not\in\Q$. \footnote{Note that the case $a_2^2=4a_4a_0$ is already excluded by assumption that $\E$ is non-trivial.} Suppose $\E$ is of that form, then by Proposition \ref{ptnontorsion} the points $P_x$ constructed previously is 2-torsion on $C_x$ for almost every $x\in a_4\Q^{*2}$
\end{rem}

\begin{proof}

We can assume that $a_1=a_3=0$. For these surfaces, the application $(x,y,t)\mapsto x$ is a fibration in genus 1 curves, infinitely many of which (in fact every fiber at $x\in a_4\Q^{*2}$ except a finite number of them)) admits a structure of group and a point of infinite order.
This shows the density of rational points of $\E(\Q)$.
\end{proof}

In the two next sections, we give other arguments showing the density of the rational points in more generality.

\subsection{A conic-bundle structure}
\begin{thm}
Let $\E$ be a rational elliptic surface of Weiestrass equation $$y^2=x^3+A(T^2-\alpha)(T^2-\beta)x,$$
where $A,\alpha,\beta\in\Q$.
Then the rational points are Zariski-dense.
\end{thm}

\begin{proof}
By changing variables $X=(T^2-\alpha)x$ and $Y=(T^2-\alpha)^2y$, one obtain the equation 
$$Y^2=(T^2-\alpha)X^3+A(T^2-\beta)(T^2-\alpha)^4X$$
which is isomorphic to $$Y^2=(T^2-\alpha)X^3+A(T^2-\beta)X.$$ 
A reshuffle of the terms permits to obtain the following equation for $\E$  $$Y^2-T^2(X^3-AX)+(\alpha X^3+XA\beta)=0$$ which is a conic bundle.
This bundle has less than 6 singular fibers. Corollary 8 of the article of Koll\'ar and Mella \cite{KollMell} thus shows unirationnality of $\E$. Therefore, the rational points are dense.
\end{proof}

\subsection{Density on isotrivial elliptic surfaces with $j=1728$}

As a conclusion for this section, we show the density of rational points on every isotrivial rational elliptic surfaces with $j$-invariant 1728.

\begin{thm}\label{density1728}
Let $\E$ be an isotrivial rational elliptic surface with $j(T)=1728$.
Then the rational points $\E(\Q)$ are Zariski-dense.
\end{thm}

\begin{proof}
Let $\E$ be an isotrivial rational elliptic surface with $j$-invariant $j(T)=1728$.

Recall Theorem \ref{Isk} due to Iskovskikh that says that a rational elliptic surface has a minimal model which is either
 a conic bundle of degree 1
 or a del Pezzo surface.

Let $X$ be a minimal model of $\E$. As a corollary of Lemma \ref{thmDP1}, $X$ is never a del Pezzo surface of degree 1. Indeed, the discriminant of $\E$ is $\Delta(T)=-2^6\cdot A(T)^3$ so every fiber at a factor of $A(T)$ has a reduction of Kodaira type $III$, $I_0^*$ or $III^*$, and such singular fibers are reducible.

In the case where $X$ is a conic bundle of degree 1, \cite{KollMell} show unirationality of $X$, and thus the density of its rational points. Therefore, it is also the case for $\E$.
In the case where $X$ is a del Pezzo surface of degree $\geq3$, \cite{Manin} shows unirationality of $X$ and $\E$.

We only have to consider the case where $X$ is a del Pezzo surface of degree 2. In this case, we have the following sections on $\E$ :
\begin{enumerate}
\item The section of the points at infinity $[0,y,0,0]$.
\item The section of Proposition \ref{ptnontorsion} $[x,\frac{-b}{u^2},\frac{1}{u},1]$ where $u=0$, $b=\sqrt{a_4x}$ and $x\in a_4\Q^{*2}$.
\item The section $[0,0,t,1]$.
\end{enumerate}

If the contraction two of them gives a del Pezzo surface of degree 2, then the image of the third is a rational curve. If it is an exceptional curve, we can contract is to obtain a del Pezzo of degree 3, on which the rational point are dense. If the image is not an exceptional curve, it allows all the same to find an infinity of points on $X$. Therefore, some of them are not on an exceptional curve nor on a distinguished quartic. We can thus use the work of Salgado, Testa and V\'arilly-Alvarado \cite{STVA} which shows unirationality and density of rational points on $X$ and $\E$. 
\end{proof}

\section{Root number on the fibers of isotrivial rational elliptic surfaces}\label{rootnumberIRES}

As seen in Section \ref{rootnumber}, the root number is conjecturally equal to the parity of the geometric rank. It can thus be used as a substitute and becomes useful in the study of the density of rational point, especially when no geometric argument is known.

Let $\E$ be an isotrivial elliptic surface\footnote{The non-isotrivial case is already studied in \cite{Desjardins1}}. We study here the variation of the root number of the fibers $\E_t$, more precisely the cardinality of the sets $$W_\pm(\E)=\{t\in\Q \ \vert \ W(\E_t)=\pm1\}.$$

If $\sharp W_-=\infty$, we can conclude the density of the rational points conditionally to the parity conjecture. 

We restrict ourselves to specific surfaces of quartic (j-invariant $=$ $1728$) and sextic twists ($j=0$), since the quadratic twist case is detailled in another paper \cite{Desjardins3}.

\subsection{Case $j(T)=0$}\label{IRES0root}

Let $\E$ be a rational elliptic surface described by the Weierstrass equation
$$\E : y^2=x^3+f(T),$$
where $f\in\Q[t]$ is such that $\deg f\leq 6$ and is sixth-power free.

A general geometric argument to show the density of the rational points, like those presented in the previous sections, is still not known. However, there are partial results. In \cite{Ulas} and \cite{Ulas2}, Ulas gives various conditions for the rational points on $\E^f$ to be dense: 1) when $\E^f$ is related to a del Pezzo surface of degree 1 and that a certain section on $\E^f$ is non-torsion, 2) when $f$ is a monic polynomial of degree six and $f$ is not even. Jabara generalizes this second Ulas' work in \cite[Theorem C]{Jabara} and treats the case with $f(T)$ monic.


We chose to study a surface given by an equation of the form \begin{equation}\label{twistssextiques}y^2=x^3+AT^6+B\end{equation} as a sequel of \cite[Theorem 2.1]{VA} which shows that the variation of the root number of the fibers of a rational elliptic surface of the form $$y^2=x^3+F(T)$$ where $F$ has a primitive factor $f_i$ such that $\mu_3\not\subseteq \Q[T]/f_i$ where $\mu_3$ is the group of the third roots of unity. 
A natural example of a polynomial not obeying this condition is \begin{equation}\label{formepreferee}F(T)=C(3A^2T^6+B^2).\end{equation}
Our Theorem \ref{conditions0} is thus the natural continuation of the work of V\'arilly-Alvarado, in particular of \cite[Theorem 1.1]{VA}.

In broad terms, the proof used in V\'arilly-Alvarado's article is based on the fact that the root number of the fiber $\E_{t=\frac{m}{n}}$, $m,n\in\Z$ coprime, is given by the formula (\cite[Prop. 4.8]{VA}) : 
\begin{equation}\label{formulern}W(\E_t)=-R(t)\prod_{p^2\mid F(m,n)\atop p\geq5}\begin{cases}
1&\text{if }v_p(F(m,n))\equiv0,1,3,5\mod6\\
\left(\frac{-3}{p}\right)&\text{if }v_p(F(m,n))\equiv2,4\mod6\end{cases}\end{equation}
where
$$R(t)=W_2(\E_t)\left(\frac{-1}{F(m,n)_2}\right)W_3(\E_t)(-1)^{v_2(F(m,n))},$$
where $F(m,n)=F(\frac{m}{n})n^{\deg F}$. It relies essentially on making the product over $p\geq5$ vary.
Families of sextic twists with $F(T)$ of the form (\ref{formepreferee}) have the property that whenever $p\mid F(m,n)$ and $p\nmid C$ for $p\geq5$ then
$$F(m,n)=C(3A^2m^6+B^2n^6)\equiv0\mod p$$ and thus $$\left(\frac{Bn^3}{Am^3}\right)^2\equiv-3\mod p,$$
forcing the terms in the product in the formula (\ref{formulern}) to be always equal to $+1$ except maybe for $p\mid C$.
This allows to prove the following:

\begin{thm}\label{conditions0}
Let $\E_{A,B,C}$ be an elliptic surface described by the Weierstrass equation $$\E_{A,B,C}:y^2=x^3+aT^6+b,$$
where $a,b\in\Z$.

Then, the function of the root number of the fibers is constant, except for the surfaces of the form $\E_{A,B,C}:y^2=x^3+C(3A^2T^6+B^2)$ such that the integers $A,B$ coprime and $C$ fulfill one of the conditions of Lemma \ref{lemmej0w3}, and one of Lemma \ref{lemmej0w2}.
\end{thm}

\begin{proof}
Put $C=\pgcd(a,b)$.
If $3a/b$ is not a rational square, then by \cite[Thm 2.1]{VA} the two sets $$W_{\pm}=\{t\in\P^1\ \vert \ W(\E_t)=\pm1\}$$ are both infinite, or in other words, the root number of the fibers of $\E$ is non-constant. Suppose thus $3A/B$ is a rational square, that i.e. there exists $A,B\in\Z$ such that $3A^2=\frac{a}{C}$ and $B^2=\frac{b}{C}$.

For each $t\in\Q$, let $(m,n)$ be the pair of coprime integers such that $t=\frac{m}{n}$, and let $\E_{m,n}$ denote the elliptic curve $\mathscr{E}_{m,n}: y^2=x^3+C(3A^2m^6+B^2n^6),$ which is isomorphic to $\E_t$. Observe that $\E_{m,n}$ and $\E_t$ must then have the same root number.
Put $P(m,n)=C(3A^2m^6+B^2n^6)$.

Thereafter, we will use the following notations to put together similar terms in the formula (\ref{formulern}):

$$\omega_2(t):=W_2(\E_{t})\left(\frac{-1}{P(m,n)_2}\right)\qquad\text{and }\omega_3(t):=W_3(\E_{t})(-1)^{v_3(P(m,n))}.$$ and
$$\mathscr{P}(t):=\prod_{p^2\mid\alpha\atop p\geq5}\begin{cases}
1&\text{if }v_p(P(m,n))\equiv0,1,3,5\mod6\\
\left(\frac{-3}{p}\right)&\text{if}v_p(P(m,n))\equiv2,4\end{cases}$$


First, note that for any choice of $A,B,C$, the function $\mathscr{P}(t)$ is a constant. Indeed, for any prime $p$ dividing a certain value $3A^2m^6+B^2n^6$, one has
$$\left(\frac{Bn^3}{Am^3}\right)^2\equiv-3\mod p.$$
We thus have $\mathscr{P}=(-1)^{\sigma}$,
where
\begin{displaymath} 
\sigma=\#\{p\text{ such that }p^2\vert C\text{, }v_p(C)\equiv2,4\mod6\text{ and }p\equiv2\mod3\},
\end{displaymath}

Moreover, note that the three functions are independent to each other namely:
\begin{description}\item[the function $\omega_2$] depends on $v_2(P(m,n))\mod4$ and $P(m,n)_2\mod4$
\item[the function $\omega_3$] depends on $v_3(P(m,n))\mod6$, and $P(m,n)_3\mod 9$.
\end{description}

Therefore, if one of the values $\omega_2$ or $\omega_3$ non-constant, then the global root number is non-constant.

This proves that the root number is non constant, except for the surfaces such that $A,B,C$ fulfill one of the conditions of Lemma \ref{lemmej0w2}, and one of Lemma \ref{lemmej0w3}.

\end{proof}

\begin{rem}
The independance of $\omega_2$ and of $\omega_3$ is also given by the Helfgott's formula for the average root number \cite[Proposition 7.2]{Helfgott}.
\end{rem}

This allows us to compute the value of the constant root number in each of the special cases.

\begin{ex}
Suppose that $A=B=1$.  
Let $\E_{1,1,C}$ be the elliptic surface defined by the equation $$\E:y^2=x^3+C(T^6+1).$$
By looking at the tables \ref{constantplus3}, \ref{constantmoins3}, \ref{constantplus2} and \ref{constantmoins2},  we have that the function $\omega_3(t)$ is constant when $t$ runs through $\Q$ if and only if
\begin{enumerate}
\item (for $\omega_3(t)=-1$)
\begin{enumerate}
\item$v_3(C)\equiv0\mod6$ and $C_3\equiv1\mod9$,
\item$v_3(C)\equiv2\mod6$ and $C_3\equiv2\mod9$, 
\item or $v_3(C)\equiv5\mod6$ and $C_3\equiv2,8\mod9$,
\end{enumerate}
\item (for $\omega_3(t)=+1$)\begin{enumerate}
\item $v_3(C)\equiv2\mod6$ and $C_3\equiv1,7\mod9$ 
\item $v_3(C)\equiv3\mod6$ and $C_3\equiv8\mod9$
\item $v_3(C)\equiv5\mod6$ and $C_3\equiv7\mod9$
\end{enumerate}\end{enumerate}
and that the function $\omega_2(t)$ is constant when $t$ runs through $\Q$ (and $\omega_2(t)=+1$) if and only if $v_2(C)\equiv1,3,5\mod 6$.


If we suppose that $C$ is less or equal to $100$, we find only the following values for which the root number is constant:
\begin{itemize}
\item if $C=10,18,46,82$ the root number is -1.
\item if $C=90$ the root number is +1.
\end{itemize}


When the surface has negative root number on every non-singular fiber, the parity conjecture states that the rank of the fibers of this elliptic surface should be always positive. For the surfaces on which the root number is $+1$, however, it is not possible to conclude anything about the density of rational points from the study of the variation of the root number.

In the case of $C=90$, the surface $\E_{1,1,90}$ has no section defined over $\Q$, and so as far as we know the density of the rational points is still an open question.
\end{ex}

\subsection{Case $j(T)=1728$}\label{IRES1728root}

The density of rational points on certain elliptic surfaces of the form $\E:y^2=x^3+g(T)x$ is garanteed by the construction of a section for $\E$ done in Section \ref{constructionsection}. However, there are surfaces such that this section is not of infinite order. This happens in particular when $g(T)=AT^4+B$. This case fails as well to satisfy the hypotheses of \cite[Theorem 2.3]{VA} and it is thus possible that the root number is constant. By the parity conjecture, an elliptic surface with constant root number always equal to $+1$ is such that every fiber has even rank, thus although the following result doesn't give new density result, it still give us some interesting (conditional) information about the distribution of the rank in the family of the fibers.

\begin{thm}\label{conditions1728}
Let $\mathscr{F}_{A,B,C}$ be an elliptic surface represented by the Weierstrass equation $$\mathscr{F}_{A,B,C}:y^2=x^3+C(A^2T^4+B^2)x,$$
where $A,B,C\in\Z$ and $\gcd(A,B)=1$.

Then, the function $t\rightarrow W(\E_t)$ of the root number of the fibers is constant, except for the specific surfaces such that $A,B,C$ fulfill one of the conditions of Lemma \ref{j1728w2}, and one of Lemma \ref{j1728w3}.
\end{thm}

\begin{proof}
Let $A,B,C\in\Z$ be such that $\gcd(A,B)=1$. Let us write $\mathscr{F}=\mathscr{F}_{A,B,C}$.
For each $t\in\Q$, let $(m,n)$ be the pair of coprime integers such that $t=\frac{m}{n}$, and let $\mathscr{F}_{m,n}$ denote the elliptic curve $\mathscr{F}_{m,n}: y^2=x^3+C(3A^2m^4+B^2n^4)x,$ isomorphic to $\mathscr{F}_{t}$.

Put $P(m,n)=C(3A^2m^4+B^2n^4)$.
It is not very difficult to see that the root number is given by the formula\footnote{This formula is shown in \cite{Desjardins3}.} $$W(\mathscr{F}_t)=-W_2(t)W_3(t)\left(\frac{-2}{P(m,n)_2}\right)\prod_{p^2\mid P(m,n)\atop p\geq5}\begin{cases}
1&\text{if }v_p(P(m,n))\equiv0,1,3,5\mod6\\
\left(\frac{-1}{p}\right)&\text{if}v_p(P(m,n))\equiv2,4\end{cases}.$$

Thereafter, we will use the following notations to congregate similar terms:

$$\omega_2(t):=W_2(\mathscr{F}_{m,n})\left(\frac{-2}{P(m,n)_2}\right)\qquad\text{and }\qquad\omega_3(t):=W_3(\mathscr{F}_{m,n})$$ and
$$\mathscr{P}(t):=\prod_{p^2\mid P(m,n)\atop p\geq5}\begin{cases}
1&\text{if }v_p(P(m,n))\equiv0,1,3,5\mod6\\
\left(\frac{-1}{p}\right)&\text{if }v_p(P(m,n))\equiv2,4\mod6\end{cases}$$


First, note that for any choice of $A,B,C$, the function $\mathscr{P}(t)$ is a constant.  Indeed, for any prime $p$ dividing a certain value $A^2m^4+B^2n^4$, one has
$$\left(\frac{Bn^2}{Am^2}\right)^2\equiv-1\mod p.$$ We have thus $\mathscr{P}=(-1)^{\sigma}$,
where
\begin{displaymath} 
\sigma=\#\{p\text{ such that }p^2\vert C\text{, }v_p(C)\equiv2,4\text{ and }p\equiv1\mod4\},
\end{displaymath}

Moreover, note that the three functions depends on independent parameters, namely:
\begin{description}\item[the function $\omega_2$] depends on $v_2(P(m,n))\mod4$ and $P(m,n)_2\mod4$
\item[the function $\omega_3$] depends on $v_3(P(m,n))\mod6$, and $P(m,n)_3\mod 9$.
\end{description}

Therefore, if one of the values $\omega_2$ or $\omega_3$ non-constant, then the global root number is non-constant.

Therefore, the root number is non constant, except for the surfaces such that $A,B,C$ fulfill one of the conditions of Lemma \ref{lemmej0w2}, and one of Lemma \ref{lemmej0w3}.

\end{proof}

\begin{ex}

Suppose that $A=B=1$. Let $\mathscr{F}_{1,1,C}$ be the elliptic surface given by the equation $$\mathscr{F}_{1,1,C}:y^2=x^3+C(T^4+1)x.$$
By looking at the formula of Lemma \ref{j1728w3} as well as Tables \ref{xconstantplus2} and \ref{xconstantmoins2}, we find that the function $\omega_3(t)=W_3(t)$ is always constant when $t$ runs through $\Q$ and its values is
\begin{enumerate}
\item $W_3(\E_t)=+1$ if $v_3(C)\equiv1,2,3\mod 4$
\item $W_3(\E_t)=-1$ if $v_3(C)\equiv0\mod4$,
\end{enumerate}
and the function $\omega_2(t)$ is constant and equal to $-1$ if and only if
\begin{enumerate}
\item $v_2(C)\equiv0\mod4$ and $C_2\equiv1,7,9,11\mod16$
\item $v_2(C)\equiv1\mod4$ and $C_2\equiv3\mod8$
\item $v_2(C)\equiv2\mod4$ and $C_2\equiv5,7,9,15\mod16$
\item $v_2(C)\equiv3\mod4$ and $C_2\equiv5\mod8$ 
\end{enumerate}
This makes quite a lot of possibilities for $\mathscr{F}_{1,1,C}$: for $C\leq20$ we have the following:
\begin{enumerate}
\item the root number of every fiber is $+1$ if $C=1,6,7,11,16,17$
\item the root number of every fiber is $-1$ if $C=9,20$.
\end{enumerate}
However, the density of rational points holds all the same on every surface $\mathscr{F}_{A,B,C}$ regardless of the variation of the root number by Theorem \ref{density1728}.
\end{ex}


\section{Non-isotrivial rational elliptic surfaces}\label{noniso}

\subsection{Known results}

In the ph.D thesis of the author \cite{Desjardinsthese} and in \cite{Desjardins1}, we prove the following theorem. This work is based on a preprint of Helfgott \cite{Helfgott}, revisited and completed with a different approach.

\begin{thm}\label{inconditionnelles}Let $\mathscr{E}$ be a rational elliptic surface given by the equation
\begin{displaymath}
\mathscr{E}: y^2=x^3+F(u,1)x+G(u,1),
\end{displaymath}
where $F$ and $G$ are homogeneous polynomials of degree respectively 4 and 6 defining a minimal model. We suppose that $\E$ is non-isotrivial, and thus in particular $FG\not=0$. Define the two following polynomials associated to $\E$: 
\begin{itemize}\item $\Delta_{\E}(U,V)=4F(U,V)^3+27G(U,V)^2=\prod^{s}_{i=0}{P_i^{m_i}(U,V)}$ (the homogeneous discriminant of $\E$) 
\item and $M_\E(U,V)=\prod_{i\in\mathscr{M}_\E}{P_i(U,V)}$ where $\mathscr{M}_\E=\{i$ such that $P_i\nmid F\}$ (the product of polynomials coming from places of multiplicative reduction).\end{itemize}

Suppose that every $P\mid\Delta_\E$ verifies the squarefree conjecture and every $P\mid M_\E$ verifies Chowla's conjecture.

Then the sets $W_\pm$ are both infinite.
\end{thm}

This means that $\Delta_\E$ needs to verify the squarefree conjecture, and that $M_\E$ needs to verify Chowla's conjecture. Those conjectures are known to hold in the following cases:

\begin{thm}
Let $P\in\Z[X,Y]$ be a homogeneous polynomial.
\begin{enumerate}
\item (Greaves \cite{Greaves}) The squarefree conjecture holds if $\deg P_i \leq 6$.
\item (Helfgott \cite{HelfChowla}, Lachand \cite{Lachandthese}) Chowla's conjecture holds if $\deg P\leq3$ or (Green-Tao \cite{GT}) if $P$ is a product of linear factors;
\end{enumerate}
\end{thm}

The following proposition classifies all the rational elliptic surfaces on which Theorem \ref{inconditionnelles} is unconditional.
\begin{propo}\label{DP1inconditionnellesHelfgott}

Let $\E$ be a non-isotrivial rational elliptic surface given by the equation:
$$\E:y^2=x^3+F(T,1)X+G(T,1),$$
where $F$ and $G$ are homogeneous polynomials 
of degree respectively $4$ and $6$.  

We suppose that $\E$ respects one of the following properties:
\begin{enumerate}
\item $\mathscr{M}=\emptyset$;
\item the places in $\mathscr{M}$ are all rational;
\item $\mathscr{M}=\{P\}$ where $P\in\Z[T]$ is a polynomial of degree 3;
\item $\mathscr{M}=\{P_1,P_2\}$ where $P_1,P_2\in\Z[T]$ are polynomials of degree respectively 1 and 2;
\item $\mathscr{M}=\{\frac{1}{T},P_2\}$ where $P_2\in\Z[T]$ is a polynomial of degree 2. 
\end{enumerate}

Then the sets $W_\pm$ are both infinite.
\end{propo}

\begin{rem}
There are examples of rational elliptic surfaces of each of the case of the list.

When $\mathscr{M}=\emptyset$, the surface obtained by the contraction of the canonical section never is a del Pezzo surface of degree 1. Indeed, an elliptic surface with no place of multiplicative reduction admits automatically a place of potentially multiplicative reduction. In this case, Corollary \ref{thmDP1} gives us that $\E$ does not come from a degree 1 del Pezzo surface. Each of the four last classes of rational elliptic surfaces contains del Pezzo surfaces of degree 1.

\end{rem}

\begin{rem}
The geometric arguments presented in the section \ref{dPgeometrique} prove the density in certain cases on which it is not possible to apply unconditionnally the work of Helfgott.
In particular, Proposition \ref{cubique} requires that there exists a rational place of type $I_m^*$, $II^*$, $III^*$ $IV^*$ or $I_0^*$.
\end{rem}
\begin{proof} (of Proposition \ref{DP1inconditionnellesHelfgott})

Let $B_\E$ and $M_\E$ be the polynomials such that 
\begin{enumerate}
\item $B_\E$ is the product of the polynomials associated to the places of bad reduction of $\E$ that are not of type $I_0^*$,
\item $M_\E$ the product of the polynomials associated to the places of multiplicative reduction of $\E$.
\end{enumerate}

Theorem \ref{inconditionnelles} and the parity conjecture show the variation of the root number on the fibers when $\E$ is a non-isotrivial surface whose polynomial $B_\E$ and $M_\E$ are such that
\begin{enumerate}
\item $B_\E$ verifies the squarefree conjecture,
\item $M_\E$ verifies Chowla's conjecture.
\end{enumerate}


If these exists no place of multiplicative reduction on $\E$, we have $M_\E=1$. Thus there is no need to consider Chowla's conjecture. Moreover, the irreducible factors of $\Delta$ appear with the exponent $\geq2$. They are of degree $\leq6$. Therefore, squarefree conjecture holds.

Suppose now that $\E$ admits a place of multiplcative reduction on $\E$.

Let be the following minimal Weierstrass model for $\E$:
$$y^2=x^3+F(T,1)x+G(T,1),$$
where $F,G\in\Z[U,V]$ are homogeneous polynomials of degree $4$ and $6$ respectively.
Let $C$, be the largest primitive polynomial such that $C\mid F$ and $C^2\mid G$. We write $F=aCF_1$ and $G=bC^2G_1$ where $F_1$ and $G_1$ are primitive polynomial and $a,b\in\Z$ are constants.
Let $R:=\pgcd(F_1,G_1)$. Observe that the polynomial $R$ splits by construction. We write $F=aCRF_2$ and $G=bC^2RG_2$ where $F_2,G_2\in\Z[X,Y]$ are suitable polynomials.
The discriminant can be written $$\Delta=C^3R^2(4a^3RF_2^3+27b^2CG_2^2).$$
When the surface is non-isotrivial, if there exists $P$ a polynomial such that $P^4\mid F$, then $P^6\nmid G$, and thus $\mathrm{ord}_PC\leq3$.
Observe that $R$, $C$, $F_2$ and $G_2$ verify squarefree conjecture as their degrees are $\leq6$. 

We define $M_o=(4RF_2^3-27CG_2^2)$ and we observe that $\Delta=C^2R^3M_0$. 
The polynomial $M_o$ is a product of powers of polynomials associated to places of multiplicative or additive reduction.
It is possible that $M_o$ is divisible by the polynomials associated to places of additive reduction: the factors of $C$ or $R$. 
For $F=P_1^{\alpha_1}\dots P_r^{\alpha_r}$ (the decomposition of $F$ in irreducible factors) there exist integers $\beta_i\in\mathbb{N}$ such that $$M_1=\frac{M_o}{P_1^{\beta_1}\dots P_r^{\beta_r}}.$$

\end{proof}



\subsection{Rational elliptic surfaces with no place of multiplicative reduction}\label{sectionsansIm}

\begin{propo}
Let $X$ be a non-isotrivial rational elliptic surface with no place of reduction of type $I_m$. Then $X$ can be described by one of the following equations:
\begin{equation}\label{propsurface1}
\E_1:y^2=x^3+aL_1^2Qx+bL_1^3QM,
\end{equation}
where $Q=\frac{cL_1^2-27b^2M^2}{4a^3}$; and
\begin{equation}\label{propsurface2}
\E_2:y^2=x^3+aL_1^2L_2L_3x+bL_1^3L_2^2L_3,
\end{equation}
where $L_1=4a^3L_3-27b^2L_2$.
We have $a,b\in\Z$, $L_1$, $L_2$, $L_3$ and $M$ linear polynomials and $Q$ a quadratic polynomial.

\end{propo}

\begin{rem}
The homogeneous and one-variable versions of the conjectures hold on the surfaces $\E_a$ and $\E_b$. Indeed, Chowla's conjecture is true since $M_\E=1$, and as every irreducible factors of the coefficients are linear, squarefree conjecture also holds. 
\end{rem}

\begin{rem}
In the first case, the places of bad reduction are those associated to $L_1$ (of type $I_2^*$), and those associated to the irreducible factors of $Q$ (of type $II$).

In the second case, we have three rational places of bad reduction: the one associated to $4a^3L_2-27b^2L_3$ has type $I_1^*$, the one associated to $L_2$ has type $III$ and the one associated to $L_3$ has type $II$.
\end{rem}

\begin{proof}
Let $\E$ be the rational elliptic surface associated to $X$, given by the Weierstrass equation:
$$\E: y^2=x^3-27c_4(T)x-54c_6(T),$$
where $c_4(T),c_6(T)\in\Z[T]$ have degree respectively less than or equal to $4$ and $6$. Let $\Delta$ be the discriminant of $\E$. This surface has a place of reduction of type $I_m^*$ because the invariant $j=\frac{c_4^3}{\Delta}$ admits necessarily a pole (at a irreducible polynomial $P\in\Z[T]$ or at $\frac{1}{T}$).

Each fiber at $t=\frac{m}{n}$ is given by the equation :
$$\E_t=\E_{m,n}: y^2=x^3+n^{4-\deg c_4}c_4\left(\frac{m}{n}\right)x+n^{6-\deg c_6}c_6\left(\frac{m}{n}\right).$$

Suppose the $P$ is the polynomial associated to a place of reduction type $I_m^*$, then we write
$F=P^2F_1$ and $G=P^3G_1$ for $F_1,G_1$ polynomials. We have that $P\mid (F_1^3-G_1^2)$.
We need to have $\deg P =1$. Indeed, if $\deg P=2$, then $G_1$ and $F_1$ are constant and thus $\E$ is isotrivial. The case where $\deg P >2$ is not possible because we would have $\deg G > 6$.
Therefore, a non-isotrivial rational elliptic surface with no place of type $I_m$ admits a rational place of reduction $I_m^*$. We have $\deg P=1$, $\deg F_1=2$ and $\deg G_1=3$.

The case where $(F_1,G_1)=1$ is not possible. Indeed, we would have $P^6\mid F_1^3-G_1^2$ and the surface would be isotrivial. Therefore, $F_1$ and $G_1$ have a common factor, that we will denote by $A$. We write $F_1=AF_2$ and $G_1=AG_2$ for convenient polynomial $F_2$ and $G_2$. We have $\Delta=P^6A^2(AF_2^3-G_2^2)$. The reduction at $A$ is thus additive.

Suppose $\deg A=2$. In this case, if $(A,G_2)=1$, we have $$A=\gamma P^2+G_2^2.$$ If $(A,G_2)=A_2$ for a linear polynomial $A_2$, then we have $$P=\frac{A_1-A_2}{\gamma}.$$
Suppose $\deg A=1$. If $A\mid G_2$, we must have 
$$A=\frac{F_2^3-\gamma P^3}{G_3^2}.$$ 
However, there exist no polynomial $A,P,G_2,F_2\in\Z[T]$ with this property. Indeed,
by imposing a linear change $v=P(t)$, and puting $\nu=\frac{u}{v}$, we are lead to solve $$4a^3F_2(\nu)^3+27b^2A(\nu)M(\nu)^2=c.$$
As $F_2\not=P$, $F_2(\nu)$ is non-constant. Let $u_0$ such that $F_2(u_0)=0$. We have 
$$27b^2A(u_0)G_3(u_0)^2=c\not=0.$$
By deriving at $u_0$, we obtain :
$$2A(u_0)G_3(u_0)G_3'(u_0)+A'(u_0)G_3(u_0)^2=0.$$
When we derive another time, we have :
$$2A(u_0)G_3'(u_0)^2+2A(u_0)+4A'(u_0)G_3'(u_0)G_3(u_0)+A''(u_0)G_3(u_0)^2=0.$$
We observe that $G_3$ is linear. We have:
$$2A(u_0)G_3'(u_0)+4A'(u_0)G_3(u_0)=0.$$
Therefore, $A$ is proportionnal to $G_3$. For all $P\in\Z[T]$ linear, the polynomial $P(T)^3-c$ has no double root. Thus, $F_2$ has to be constant. Therefore $G_3,F_2,A$ and $P$ are proportional to each other and the surface $\E$ is isotrivial.

If $A\nmid G_2$, we must have the equality
$$A=\frac{\gamma P^4+G_2^2}{F_2^3}.$$ 
By a similar argument as in the previous case, this is not possible either.
\end{proof}

\subsection{Geometric arguments}\label{dPgeometrique}\label{piste}

In this section we prove unconditionally the density on many more elliptic surfaces, not necessarily isotrivial. Moreover Helfgott's paper does not prove unconditionally the variation of the root number for those surfaces.

\begin{propo}\label{cubique}
Let $\E$ be a elliptic surface given by the equation
\begin{equation}\label{torduelineaire}
\E: y^2=x^3+L^2Qx+L^3C,
\end{equation}
where $L,Q,C\in\mathbb{Z} [u,v]$ have respective degree 1, 2 and 3.
Then the surface is $\Q$-unirational.
In particular, $\E(\Q)$ is Zariski-dense.
\end{propo}

\begin{rem} The polynomial $L$ of the surface $\E$ in this proposition is such that $L^6\mid\Delta$. As we chose a minimal Weiestrass model for $\E$, this means that the reduction at $L$ has type $I_0^*$, $II^*$, $III^*$, $IV^*$ or $I_m^*$. Conversly, if we consider a surface with a rational place of one of these types, we can find an equation of the form (\ref{torduelineaire}). We deduce directly the following corollary:
\end{rem}

\begin{coro} 
If a rational elliptic surface $\E$ has a rational place of type $I_0^*$, $II^*$, $III^*$, $IV^*$ or $I_m^*$, then the rational points of $X$ are Zariski-dense.

In particular, if $\E$ is a non-isotrivial elliptic surface with no place of multiplicative reduction, then its rational points are dense.
\end{coro}

\begin{proof}
Let $S$ be an elliptic surface given by the equation
\begin{displaymath}
S: y^2=x^3+L(t,1)^2Q(t,1)x+L(t,1)^3C(t,1),
\end{displaymath}
where $L,Q,C\in\mathbb{Z}[u,v]$ have respective degree 1, 2 and 3. Note that this surface is rational.
We study the surface which is birational
\begin{equation}\label{etoile}
\Big(\frac{y}{L^3}\Big)^2=\Big(\frac{x}{L^2}\Big)^3+\frac{Q}{L^2}\Big(\frac{x}{L^2}\Big)+\Big(\frac{C}{L^3}\Big).
\end{equation}

We can suppose that $L(u,v)=v$ (otherwise, we do a linear change on $u,v$). Put $t=\frac{u}{v}$, $x'=\frac{x}{v^2}$ and $y'=\frac{y}{v^3}$, whose inverse transformation is $x=x'v^2$, $y=y'v^3$, $u=tv$. By this change of variables, (\ref{etoile}) becomes
\begin{displaymath}
S': y'^2=x'^3+q(t)x'+c(t)\subset\mathbb{P}^3,
\end{displaymath}
with $Q(t,1)=q(t)$ and $C(t,1)=c(t)$, which is a cubic surface with a finite number of singular points.

Note that on a cubic surface which is not a cone on a cubic curve, the existence of a rational point is equivalent to the density of the rational points. This is shown by Kollar \cite{Kollarcubique}, generalizing the work of Segre and Manin \cite{Manin}.

From a geometric point of view, this surface is obtained by the contraction of two exceptional curves. For a surface obtained by the successive blow-down of two disjoint exceptional curves (which is the case of $S'$), we are guaranteed to have a rational point: the one associated to the point $[0,0,1,1]$ (which is not singular). 
\end{proof}

In the previous section, we show that a rational elliptic surface with no place of multiplicative reduction has one of the two following forms:

\begin{equation}\label{surface1}
\E_1:y^2=x^3+aL_1^2Qx+bL_1^3QM
\end{equation}
and
\begin{equation}\label{surface2}
\E_2:y^2=x^3+a(4a^3L_3-27b^2L_2)^2L_2L_3x+b(4a^3L_3-27b^2L_2)^3L_2^2L_3,
\end{equation}
where $a,b\in\Z$, $L_1$, $L_2$, $L_3$ and $M$ linear polynomials and $Q$ a quadratic polynomial. In the first case, we impose moreover that $M$ is such that $M^2=\left(\frac{L_1^2-4a^3Q}{27b^2}\right)$.

On surface $\E_1$, the places of bad reduction are those associated to $L_1$, of type $I_2^*$, and those associated to the irreducible factors of $Q$, of type $II$.

On surface $\E_2$, we have three rational places of bad reduction - the one associated to $4a^3L_2+27b^2L_3$ has type $I_1^*$, the one associated to $L_2$ has type $III$ and the one associated to $L_3$ has $II$.

Therefore, the results previously presented prove the density of rational points on these surfaces. 
The work of Helfgott proves in those cases the density of rational points although under the parity conjecture which we are not using here.

There is a fourth method to show the density, at least for surface $\E_2$.
Let $\E$ be an elliptic surface and $E$ its generic fiber (that is to say $\E$ seen as an elliptic curve over $\Q[T]$). By the Shioda-Tate formula, we have $$\mathrm{rg}NS(\E_{\overline{\Q}})=2+\mathrm{rg}E(\overline{\Q}(T))+\sum_{v}(m_v-1).$$
The surfaces that we consider are obtained by blowing-up $\P^2$ at 9 points in general position, the N\'eron-Severi rank is $\mathrm{rg}NS(\E)=10$.

In the first case, Shioda-Tate formula says that $\mathrm{rg} E(\overline{\Q}(T))=4$. Unfortunately, although it gives an interesting majoration: $\mathrm{rg}(E(\Q(T)))\leq4$, this is not precise enough to conclude on the density.
There is indeed an uncertainty, except in the case where we can bound it this way : $\mathrm{rg}(E(\Q(T)))\geq1$.
It is just what happen in the second case. Indeed the Shioda-Tate formula gives $\mathrm{rg} E(\overline{\Q}(T))=1$.

We have $\E(\overline{\Q}(T))=\Z\cdot P_o$, for a certain point $P_o$. Therefore there exists $K$ a quadratic extension of $\Q$ such that $P_o\in\E(K(T))$. Indeed, if for every $\sigma\in\mathrm{Gal}(\overline{\Q}/\Q)=G_{\Q}$ we put $\sigma P_o:=\varepsilon(\sigma)\cdot P_o$ where $\varepsilon : G_\Q\rightarrow \pm1$, then 
\begin{enumerate}
\item either $\varepsilon$ is trivial and $P_o\in E(\Q(T))$,
\item or $\varepsilon$ is non-trivial and in this case, $\overline{\Q}^{\mathrm{Ker}\varepsilon}=K$, the subfield of $\overline{\Q}$ stabilised by $\varepsilon$, is a quadratic field such that $P_o\in E(K(T))$.
\end{enumerate}
One can remark, similarly as in the proof of Proposition \ref{cubique}, that $\E_2$ is birational to a cubic surface.
We use the following proposition to end the argument :

\begin{propo}

Let $S$ be a non-singular cubic surface on a number field $k$. Suppose $S$ is not a cone on a cubic curve.

\begin{enumerate}
\item If $S(k)\not=\varnothing$, then $S(k)$ is Zariski-dense.
\item Let $k_1$ be a quadratic extension of $k$. If $S(k_1)\not=\varnothing$ is Zariski-dense, then  $S(k)$ is Zariski-dense.
\end{enumerate}

\end{propo}

\begin{proof}
The first statement of the proposition is shown by Segre and Manin \cite{Manin} and by Koll\'ar \cite{Kollarcubique}. They actually prove a stronger result : if $k$ is an arbitrary field and that $S(k)\not=\varnothing$, then $S$ est $k$-unirational. When $k$ is infinite, this implies the Zariski-density of rational points.

We now show the second point of the proposition.
Let $P\in S(k_1)$. 
If $P\in S(k)$, then the rational points are dense.
Suppose the that $P\not\in S(k_1)$. Consider $D$ the line passing through $P$ and $P^\sigma$ where $\sigma$ is the automorphism of $k_1$ fixing $k$. 
If $D\subset S$, then $D(k)\subset S(k)$ and thus the set of rational points of $S$ is not empty.
Otherwise, the intersection $D\cap S$ contains three points: $P$, $P^\sigma$, and a third point which is necessarily in $S(k)$.
\end{proof}

We end the section with a result concerning smooth rational elliptic surfaces, associated to a del Pezzo surface of degree 1.
Let $X$ be a del Pezzo surface of degree 1. In general, if there exists $C_1$ and $C_2$ a pair of exceptional curves defined over $\Q$ on $X$ such that their intersection is empty, one can contract those curves to obtain $X_3$ a del Pezzo surface of degree 3. On $X_3$, the existence of a rational point garantees the Zariski-density of $X(k)$.
In what follows, we use this idea to prove the density one some other surfaces on which we find two exceptional curves with non empty intersection.

\begin{propo}
Let $X$ a del Pezzo surface of degree 1 on which lie $\mathscr{C}_1$ and $\mathscr{C}_2$ two distinct exceptional curves defined over $k$ with possibly points in common.
Then $X(k)$ is Zariski-dense.
\end{propo}

\begin{proof}
 The contraction of $\mathscr{C}$ gives $X'$ a del Pezzo surface of degree 2. We know that on these surfaces, the rational points of $X'$ are dense if $X'(k)$ contains a point which is neither on an exceptional curve nor on a distinguished quartic. Put $\mathscr{E}$ the union of the points of this quartic and of the exceptional curves. The contraction sends $\mathscr{C}_2$ on a rational curve of $X'$ that we will denote by $\mathscr{C}$. 
Note that $\mathscr{C}$ is not an exceptional curve on $X'$ because it is the blow-down of a curve which has a point in common with $\mathscr{C}_1$.
In the case where $\mathscr{C}\cap\mathscr{E}$ is finite, one can find a rational point outside of $\mathscr{E}$, and this proves the density of the rational points. 
\end{proof}

\appendix

\section{Local root number}

Let $E_\alpha$ be the (sextic or quartic) twist by a non-zero $\alpha\in\Q$ of an elliptic curve with $j=0$ or $j=1728$.
In this appendix we study in more details the two functions $\omega_2$ and $\omega_3$ defined in Section \ref{rootnumberIRES} appearing in the decomposition of the root number (of Equation \ref{formulern}). The values of those functions depend only on  $\alpha_2$ and $v_2(\alpha)$ or respectively on $\alpha_2$ and $v_2(\alpha)$.
Remember our notation: for any prime number $p$, $\alpha_p$ is the integer such that $t=p^{v_p(\alpha)}\alpha_p$.

As in section \ref{rootnumberIRES}, we restrict our attention to: $$\alpha=C(3A^2m^6+B^2n^6)\text{ (in case $j=0$)}\quad\text{ and }\quad\alpha=C(A^2m^4+B^2n^4)\text{ (in case $j=1728$)}$$ and thus study the surfaces $\E_{A,B,C}$ or $\mathscr{F}_{A,B,C}$ given by the equations: \begin{equation}\label{formespreferees}\E_{A,B,C}:y^2=x^3+C(3A^2T^6+B^2)\quad\text{ and }\quad \mathscr{F}_{A,B,C}: y^2=x^3+C(A^2T^4+B^2)x\end{equation} because those are the natural cases where the root number is likely to be constant according to \cite{VA}. In those equation $A,B,C\in\Z$ are such that $\gcd(A,B)=1$.

Let us briefly recapitulate what was done in Section 6 before we state the result. We use a formula of Varilly-Alvarado splitting the root number into three functions, $\omega_2(t)$, $\omega_3(t)$, $\mathscr{P}(t)$ corresponding to the contribution of respectively the prime numbers $p=2$, $p=3$ and $p\geq5$.

While $\mathscr{P}(t)$ is constant for a given surface of one of the forms of (\ref{formespreferees}), the functions $\omega_2$ and $\omega_3$ varies independently from each other, and hence each of them must be constant for the global root number to take always the same value over the fibers of the surface.

\subsection{The elliptic surface $\E_{A,B,C}:y^2=x^3+C(3A^2T^6+B^2)$}
Let $\E$ be an elliptic surface given by the Weierstrass equation \begin{equation}\label{surfaceeqj0}\E: y^2=x^3+C(3A^2T^6+B^2),\end{equation}
where $A,B,C\in\Z$ and $\pgcd(A,B)=1$.

Put $P(m,n)=C(3A^2m^6+B^2n^6)$ and
define the two functions, for $p=2$ or $3$, $$\omega_p(t):t=\frac{m}{n}\in\Q_p\rightarrow\{-1,+1\}$$ as the following $$\omega_3(t)=W_3(E_t)(-1)^{v_3(P(m,n))}.$$
Define the function $$\omega_2(t):=W_2(\E_t)\left(\frac{-1}{P_2(m,n)}\right).$$
\begin{lm}\label{lemmej0w3}
The function $\omega_3$ is constant if and only if one option is satisfied
\begin{enumerate}
\item $A,B,C$ are in Table \ref{constantplus3} (in which case $\omega_3(t)=+1$)
\item $A,B,C$ are in Table \ref{constantmoins3} (in which case $\omega_3(t)=-1$)
\end{enumerate}
\end{lm}


\begin{lm}\label{lemmej0w2}

The function $\omega_2$ is constant if and only if one option is satisfied:
\begin{enumerate}
\item $A,B,C$ are in Table \ref{constantplus2} (in which case $\omega_2(t)=+1$)
\item $A,B,C$ are in Table \ref{constantmoins2} (in which case $\omega_2(t)=-1$)
\end{enumerate}
\end{lm}

\begin{center}
\begin{table}[!htbp]
\begin{center}
\begin{tabular}{| c | c | c | c |}
\hline
$v_3(A)[3]$&$v_3(B)[3]$&$v_3(C)[6]$ & Additional conditions  
\\

\hline
0&0&0 & $C_3B_3^2\equiv1\mod9$\\ 
&&2&$C_3A_3^2\equiv 2\mod9$\\
&&5&$C_3A_3^2\equiv2,8\mod9$\\ \hline

1&0&0&$C_3A_3^2\equiv2,4\mod9$, $C_3B_3^2\equiv1,2,4,8\mod9$\\ 
&&1,4&$C_3\equiv1\mod 3$\\
&&3&$C_3A_3^2\equiv1,2,4,8\mod9$, $C_3B_3^2\equiv2,4\mod9$\\ 
&&2,5&$C_3\equiv2\mod3$\\ \hline
2&0&0&$C_3B_3^2\equiv2,8\mod9$\\
&&1&$C_3A_3^2\equiv1\mod9$\\
&&3&$C_3B_3^2\equiv2\mod9$\\
\hline

0&1&1&$C_3B_3^2\equiv2\mod9$\\
&&4&$C_3B_3^2\equiv2,8\mod9$\\
&&5&$C_3A_3^2\equiv1\mod9$\\ \hline

0&2&0\text{, }3& $C_3\equiv1\mod3$\\
&&1,4&$C_3\equiv2\mod3$\\
&&2&$C_3A_3^3\equiv2,4\mod9$\\
&&5&$C_3B_3^2\equiv2,4\mod9$\\ \hline
\end{tabular}
\end{center}
\caption{\label{constantplus3} Cases where $\omega_3(t)=+1$  for every fiber $\E_t$ of the surface \ref{surfaceeqj0}. 
}
\end{table}
\end{center}

\begin{center}
\begin{table}[!htb]
\begin{center}
\begin{tabular}{| c | c | c | c |}
\hline
$v_3(A)[3]$&$v_3(B)[3]$&$v_3(C)[6]$ & Additional conditions  
\\

\hline
0&0&2 & $C_3A_3^2\equiv1,7\mod9$\\ 
&&3&$C_3B_3^2\equiv 8\mod9$\\
&&5&$C_3A_3^2\equiv7\mod9$\\ \hline

1&0&0&$C_3A_3^2\equiv1,5,7,8\mod9$, $C_3B_3^2\equiv5,7\mod9$\\ 
&&1,4&$C_3\equiv2\mod3$\\
&&2,5&$C_3\equiv1\mod3$\\
&&3&$C_3A_3^2\equiv2,4\mod9$, $C_3B_3^2\equiv1,5,7,8\mod9$\\ \hline
2&0&0&$C_3B_3^2\equiv7\mod9$\\
&&3&$C_3B_3^2\equiv1,7\mod9$\\
&&4&$C_3A_3^2\equiv8\mod9$\\ \hline

0&1&1&$C_3B_3^2\equiv1,7\mod9$\\ 
&&2&$C_3A_3^2\equiv8\mod9$\\
&&4&$C_3B_3^2\equiv7\mod9$\\ \hline

0&2&0\text{, }3& $C_3\equiv2\mod3$\\
&&1, 4&$C_3\equiv1\mod3$\\
&&2&$C_3B_3^2\equiv5,7\mod9$\\
&&5&$C_3A_3^2\equiv5,7\mod9$\\ \hline
\end{tabular}
\end{center}
\caption{\label{constantmoins3} Cases where $\omega_3(t)=-1$  for every fiber $\E_t$ of the surface \ref{surfaceeqj0}. 
}
\end{table}
\end{center}

\begin{center}
\begin{table}[htb]
\begin{center}
\begin{tabular}{| c | c | c | c |}
\hline
$v_2(A)\mod3$&$v_2(B)\mod3$&$v_2(C)\mod6$ & Additional conditions  
\\

\hline
0&0&1,3,5&\\
\hline
1&0&1,3,5&\\ 
&&2&$C_2\equiv3\mod4$\\ 
&&4&$C_2\equiv1\mod4$\\
\hline
2&0&0&$C_2\equiv3\mod4$\\
&&1,3,5&\\
&&4&$C_2\equiv1\mod4$\\ \hline
0&1&1,3,5&\\
&&2&$C_2\equiv1\mod4$\\
&&4&$C_2\equiv3\mod4$\\ \hline
0&2&0&$C_2\equiv1\mod4$\\
&&1,3,5&\\
&&4&$C_2\equiv3\mod4$\\
\hline
\end{tabular}
\end{center}
\caption{\label{constantplus2} Cases where $\omega_2(t)=+1$  for every fiber $\E_t$ of the surface \ref{surfaceeqj0}. 
}
\end{table}
\end{center}

\begin{center}
\begin{table}[htb]
\begin{center}
\begin{tabular}{| c | c | c | c |}
\hline
$v_2(A)\mod3$&$v_2(B)\mod3$&$v_2(C)\mod6$ & Additional conditions  
\\

\hline
1&0&2&$C_2\equiv3\mod4$\\ \hline
2&0&0\text{, }2& $C_2\equiv1\mod4$\\ \hline
0&1&0\text{, }4&$C_2\equiv1\mod4$\\
\hline
\end{tabular}
\end{center}
\caption{\label{constantmoins2} Cases where $\omega_2(t)=-1$  for every fiber $\E_t$ of the surface \ref{surfaceeqj0}. 
}
\end{table}
\end{center}

\begin{proof}of Lemma \ref{lemmej0w3}.
Let $A$,$B$,$C$ be integers such that $\gcd(A,B)=1$. To ease the notation, let us simply write $\E=\E_{A,B,C}$.
For each fiber $\E_t$, we study instead the curve $\E_{m,n}:y^2=x^3+C(A^2m^6+B^2n^6)$ which is $\Q$-isomorphic. 
Let $P(m,n)=C(3A^2m^6+B^2n^6)$. For every $(m,n)\in\Z^2$ one has $$P(m,n)=3^{v_3(C)}C_3\left(3^{2v_3(A)+6v_3(m)+1}A_3^2m_3^6+3^{2v_3(B)+6v_3(n)}B_3^2n_3^6\right),.$$ 

Remark that according to the 3-valuations of $A$ and $B$ different situations occur. We will treat in details the case where $v_3(A)=0$ and $v_3(B)=0$. In that case, $v_3(P(m,n))$ is at least \begin{equation}v_3(P(m,n))=v_3(C)+\min \big(2v_3(A)+6v_3(m)+1, 2v_3(B)+6v_3(n)\big).\end{equation}

We make the distinction between three properties for the coprime integers $m,n$:
\begin{enumerate}[(a)]
\item if $2v_3(A)+6v_3(m)=2v_3(B)+6v_3(n)$,
then $v_3(m,n)=a$ and $P(m,n)_3\equiv C_3(3A_3^2+B_3^2)\mod 9$,
\item if $2v_3(A)+6v_3(m)<2v_3(B)+6v_3(n)$ 
 then $v_3(m,n)=a$ and $P(m,n)_3\equiv C_3B_3^2\mod 9$
\item if $2v_3(A)+6v_3(m)>2v_3(B)+6v_3(n)$,
then $v_3(m,n)=a+1$ and $P(m,n)_3\equiv C_3A_3^2$
\end{enumerate}


\noindent

In those subcases, we obtain a different formula for the function $\omega_3$, as follows.

(a) Suppose that $2v_3(A)+6v_3(m)=2v_3(B)+6v_3(n)$. One has $$v_3(P(m,n))=2v_3(B)+6v_3(n)\equiv 2v_3(B)\mod6\text{ and }P(m,n)_3\equiv C_3(3A_3^2+B_3^2)\mod 9.$$

By \cite[Lemma 4.1]{VA}, the local root number at $3$ is equal to
\begin{equation}W_3(\E_{P(m,n)})=\begin{cases}
-1 &\text{if }v_3(C)\equiv1,2\mod6\text{ and }C_3\equiv1\mod3\\
&\text{or if }v_3(C)\equiv4,5\mod6\text{ and }C_3\equiv2\mod3\\
&\text{or if }v_3(C)\equiv0\mod6\text{ and }P(m,n)_3\equiv5,7\mod9\\
&\text{or if }v_3(C)\equiv3\mod6\text{ and }P(m,n)_3\equiv2,4\mod9\\
+1 & \text{otherwise.}
\end{cases}\end{equation} 
and thus
\begin{equation}\label{biduletruc3}\omega_3(t)=\begin{cases}
+1 & \text{if }v_3(C)+2v_3(B)\equiv1,4\mod6\text{ and }C_3\equiv1\mod3\\
&\text{or if }v_3(C)+2v_3(B)\equiv2,5\mod6\text{ and }C_3\equiv2\mod3,\\
& \text{or si }v_3(C)+2v_3(B)\equiv0\mod6\text{ and }P(m,n)_3\equiv1,2,4,8\mod9\\
& \text{or si }v_3(C)+2v_3(B)\equiv3\mod6\text{ and }P(m,n)_3\equiv2,4\mod9\\
-1& \text{otherwise.}
\end{cases}\end{equation}

(b) Suppose that the coprime integers $(m,n)$ are such that $2v_3(B)+6v_3(n)<2v_2(A)+6v_3(m)$. One has $$v_3(P(m,n))\equiv v_3(C)\mod6\quad\text{ and }P(m,n)_{3}\equiv 
C_{3}B_3^2 \mod 9$$
Note that $B_{3}^2$ will take values among the congruence classes $1,4,7\mod 9$, or $1\mod3$, and so $P(m,n)_3\equiv C_3\mod3$. Moreover, in case $v_3(A)+3v_3(m)=v_3(B)+3v_3(n)$, one has $$P(m,n)_3\equiv\begin{cases}4C_3&\text{if }B_3^2\equiv1\mod9\\
7C_3&\text{if }B_3^2\equiv4\mod9\\C_3&\text{if }B_3^2\equiv7\mod9.\end{cases}$$
Thus we have
\begin{equation}\label{biduletruc3}\omega_3(t)=\begin{cases}
+1 & \text{if }v_3(C)+2v_3(B)\equiv1,4\mod6\text{ and }C_3\equiv1\mod3\\
&\text{ or if }v_3(C)+2v_3(B)\equiv2,5\mod6\text{ and }C_3\equiv2\mod3,\\
& \text{or if }v_3(C)+2v_3(B)\equiv0\mod6\text{ and }P(m,n)_3\equiv1,2,4,8\mod9\\
& \text{or if }v_3(C)+2v_3(B)\equiv3\mod6\text{ and }P(m,n)_3\equiv2,4\mod9\\
-1& \text{otherwise.}
\end{cases}\end{equation}

(c) Suppose now that $6v_3(n)+2v_3(B)>2v_3(A)+6v_3(m)+1$ (and that in particular, since $3\mid n$, then $3\nmid m$). In this case one has $v_3(P(m,n))=v_3(C)+2v_2(A)+1$ and $P(m,n)_{3}\equiv 
C_{3}A_{3}^2$. As previously, we find that with this choice of $(m,n)$, the value of $\omega_3$ is
\begin{equation}\label{chosebine3}\omega_3(t)=\begin{cases}
+1 & \text{if }v_3(C)+2v_2(A)\equiv1,4\mod6\text{ and }C_3\equiv2\mod3\\
&\text{or if }v_3(C)+2v_2(A)\equiv0,3\mod6\text{ and }C_3\equiv1\mod3,\\
& \text{or if }v_3(C)+2v_2(A)\equiv5\mod6\text{ and }P(m,n)_3\equiv5,7\mod9\\
& \text{or if }v_3(C)+2v_2(A)\equiv2\mod6\text{ and }P(m,n)_3\equiv2,4\mod9\\
-1& \text{otherwise.}
\end{cases}\end{equation}

We deduce that the function $\omega_3$ is constant in the cases listed in the lemma (and only in those cases). To achieve this, we compare the two formulas for each value of $k\mod3$.

When $v_3(A)\equiv v_3(B)\equiv0\mod3$, then 
the equation \ref{chosebine3} 
compared with \ref{biduletruc3} gives:
$$\omega_3(t)=\begin{cases}
+1 & \text{if }v_3(C)\equiv0\mod6\text{ and }P(m,n)_3\equiv2,8\mod9\\
&\text{or if }v_3(C)\equiv1,4\mod6\text{ and }C_3\equiv1\mod3,\\
& \text{or if }v_3(C)\equiv2\mod6\text{ and }P(m,n)_3\equiv5,8\mod9\\
& \text{or if }v_3(C)\equiv3\mod6\text{ and }P(m,n)_3\equiv2\mod9\\
&\text{or if }v_3(C)\equiv5\mod6\text{ and }P(m,n)_3\equiv5\mod9\\
-1 &\text{if }v_3(C)\equiv0\mod6\text{ and }P(m,n)_3\equiv7\mod9\\
&\text{or if }v_3(C)\equiv1,4\mod6\text{ and }P(m,n)_3\equiv2\mod3,\\
& \text{or if }v_3(C)\equiv2\mod6\text{ and }P(m,n)_3\equiv1\mod9\\
& \text{or if }v_3(C)\equiv3\mod6\text{ and }P(m,n)_3\equiv1,7\mod9\\
&\text{or if }v_3(C)\equiv5\mod6\text{ and }C_3A^2\equiv1,4\mod9\\
\text{non-constant}& \text{otherwise.}
\end{cases}$$

When $v_2(A)\equiv1,2\mod3$, we proceed in a similar way and obtain that the cases where the root number is constant are those listed in the Tables. This is the same method for $v_2(A)=0$ and $v_2(B)\equiv1,2\mod3$. However, we have different subcases, for instance :
\begin{description}
\item[$v_3(A)=0$ and $v_3(B)=1$]
\begin{enumerate}
\item if $3\nmid m, n$, then $v_3(m,n)=a+1$ and $P(m,n)_3\equiv C_3(A_3^2+3B_3^2)\mod 9$,
\item if $3\nmid m$ and $3\mid n$, then $v_3(m,n)=a+1$ and $P(m,n)_3\equiv C_3A_3^2\mod 9$
\item if $3\mid m$ and $3\nmid n$, then $v_3(m,n)=a+2$ and $P(m,n)_3\equiv C_3B_3^2$
\end{enumerate}
\end{description}

\end{proof}

\begin{proof}of Lemma \ref{lemmej0w2}


There is only one of $A$ or $B$ at a time that may be divisible by 2. According to which of them is (or isn't), the formula for $w(t)$ is different.

Let $(m,n)\in\Z\times\Z_{\leq0}$ be a pair of coprime integers. We have $$P(m,n)=2^{v_2(C)}C_2(2^{2v_2(A)+6v_2(m)}\cdot3A_2^2 m_2^6+2^{6v_2(n)+2v_2(B)}\cdot B_2^2n_2^6).$$
So, except if $2v_2(A)+6v_2(m)=2v_2(B)+6v_2(n)$, we have that $$v_2(P(m,n))=v_2(C)+\min(2v_2(A)+6v_2(m),2v_2(B)+6v_2(n)).$$

a) If $2v_2(B)+6v_2(n)< 2v_2(A)+6v_2(m)$, then $v_2(P(m,n))=v_2(C)+6v_2(n)\equiv v_2(C)\mod6$ and moreover $P(m,n)_2\equiv B_2^2C_2\equiv C_2\mod4$.  By \cite[Lemma 4.1]{VA}, we have 
$$W_2(\E_{t})=\begin{cases}
-1 & v_2(C)\equiv0,2\mod6\\
& v_2(C)\equiv1,3,4,5\mod6\text{ and }C_2\equiv3\mod4\\
+1 &\text{otherwise.}
\end{cases}$$
and thus


\begin{equation}\label{formule02pask}
\omega_2(t)=\begin{cases}
+1 & \text{if }v_2(C)+2v_2(B)\equiv1,3,4,5\mod6 \\
& \text{or if }v_2(C)+2v_2(B)\equiv0,2\mod6\text{ and }C_2\equiv3\mod4\\
-1 & \text{if }v_2(C)+2v_2(B)\equiv0,2\mod6\text{ and }C_2\equiv1\mod4
\end{cases}\end{equation}


b) If $6v_2(n)> 2v_2(A)$, then in particular, $2\nmid m$. Hence, $v_2(P(m,n))\equiv v_2(C)+2v_2(A)\mod6$ and $P(m,n)_{2}\equiv3C_2\mod4$.
In this case we have
$$\omega_2(t)=\begin{cases}
+1 & \text{if }v_2(C)+2v_2(A)\equiv1,3,4,5\mod6\\
& \text{or if }v_2(C)+2v_2(A)\equiv0,2\mod6\text{ and }C_2\equiv1\mod4\\
-1 & \text{if }v_2(C)+2v_2(A)\equiv0,2\mod6\text{ and }C_2\equiv3\mod4.
\end{cases}
$$
From these formulas, we now deduce the behavior of the function $\omega_2(t)$ when $6v_2(n)\not=2v_2(A)$.
For instance when $v_2(A)\equiv v_2(B)\equiv0\mod3$, then 
\begin{equation}\label{binouche}\omega_2(t)=\begin{cases}
+1& v_2(C)\equiv1,3,4,5\mod6\\
\text{non-constant}&\text{otherwise.}
\end{cases}\end{equation}
We proceed in a similar way for the case $v_2(A)\equiv1,2\mod3$.

c) If $v_2(A)\equiv v_2(B)\equiv0\mod3$, we need to proceed to a more raffined selection. Let $(m,n)$ be a pair such that $6v_2(n)=2v_2(A)$. In this case, one has $v_2(P(m,n))\equiv2\mod6$ 
and $$P(m,n)_2\equiv\begin{cases}C_2\mod16&\text{if }A_2^2m_2^6\equiv B_2^2n_2^6\mod 16\\ 3C_2\mod16& \text{otherwise} \end{cases}.$$ Observe moreover that replacing $n_2$ by $n_2'$ such that $n'\equiv n_2+8\mod16$, then the value of $P(m,n)_2\mod4$ passes from $C_2$ to $3C_2$ and vice-versa. Thus we get the formula
$$\omega_2(t)=\begin{cases}
\text{non-constant} &\text{if }v_2(C)\equiv0,4\mod6\\
+1&\text{otherwise}.
\end{cases}.$$ Therefore, comparing with the formula (\ref{binouche}), we get that the function $\omega(t)$ is constant and equal to +1 in the case where $v_2(C)\equiv1,3,5\mod6$. 
The method is similar for the case $v_2(A)=0$ and $v_2(B)\equiv1,2\mod3$.
\end{proof}

\subsection{The elliptic surface $\mathscr{F}_{A,B,C}:y^2=x^3+C(A^2T^4+B^2)x$}
Let $\mathscr{F}$ be an elliptic surface given by the Weierstrass equation \begin{equation}\label{surfaceeqj1728}\mathscr{F}_{A,B,C}:y^2=x^3+C(A^2T^4+B^2)x,\end{equation}
where $A,B,C\in\Z$ and $\gcd(A,B)=1$.

\begin{lm}\label{j1728w3}
The local root number at 3 is $$W_3(\E_t)=\begin{cases}
+1&\text{if }v_3(C)\equiv1,3\mod4\\
&\text{if }v_3(C)\equiv2\mod4\text{ and }v_3(A),v_3(B)\text{ even}\\
-1&\text{if }v_3(C)\equiv0\mod4\text{ and }v_3(A),v_3(B)\text{ even}\\
\text{non-constant}& \text{otherwise.}\end{cases}$$
\end{lm}

\begin{proof}
Let $A$, $B$, $C$ integers such that $\gcd(A,B)=1$. Let us write simply $\mathscr{F}=\mathscr{F}_{A,B,C}$.
For any $t\in\Q$ consider the pair of coprime integers $(m,n)\in\Z\times\Z_{<0}$ such that $t=\frac{m}{n}$.
For each fiber $\mathscr{F}_t$, let $\mathscr{F}_{m,n}$ be the curve given by the equation $\mathscr{F}_{m,n} : y^2=x^3+C(A^2m^4+B^2n^4)x$ which is $\Q$-isomorphic to $\mathscr{F}_t$ and thus have the same local root number (at any prime $p$). Put $$P(m,n)=C(A^2m^4+B^2n^4).$$
 The local root number at 3 of $\mathscr{F}_{m,n}$ only depends of $v_3(P(m,n))$.
For every $m,n\in\Z$ coprime, we have $$P(m,n)=3^{v_3(C)}C_3\left(3^{2v_3(A)+4v_3(m)}A_3^2m_3^4+3^{2v_2(B)+4v_3(n)}B_3^2n_3^4\right),$$
and thus 
$$v_3(P(m,n))=v_3(C)+\min(2v_3(A)+4v_3(m),2v_3(B)+4v_3(n)).$$
In case where $v_3(A)+2v_3(m)<v_3(B)+2v_3(n)$, we get the formula:
\begin{equation}W_3(\mathscr{F}_t)=\begin{cases}
-1 &\text{if } v_3(C)\equiv0\mod4\text{ and }v_3(C)+2v_3(A)\equiv2\equiv2\mod4\\
&\text{or if } v_3(C)\equiv2\mod4\text{ and }v_3(A)\text{ is odd,}\\
+1 & \text{otherwise.}
\end{cases}\end{equation}
The case $v_3(A)+2v_3(m)\geq v_3(B)+2v_3(n)$ is similar, but with the condition $v_3(C)+2v_3(B)\equiv2\equiv2\mod4$. 
Comparing those formulas, we obtain the conclusion of the Lemma.




\end{proof}

Define the function $\omega_2(t):=W_2(\mathscr{F}_t)\left(\frac{-2}{P_2(m,n)}\right)$.

\begin{lm}\label{j1728w2}


The value of the function $\omega_2(t)$ 
 is constant when $t\in\Q$ varies if and only if one option is satisfied:
\begin{enumerate}
\item $A,B,C$ are in Table \ref{xconstantplus2} (in which case $\omega_2(t)=+1$)
\item $A,B,C$ are in Table \ref{xconstantmoins2} (in which case $\omega_2(t)=-1$).
\end{enumerate}

\end{lm}

\begin{center}
\begin{table}[!htb]
\begin{center}
\begin{tabular}{| c | c | c | c |}
\hline
$v_2(A)[2]$&$v_2(B)[2]$&$v_2(C)[4]$ & Additional conditions  
\\

\hline

1&0&0&$C_2A_2^2\equiv5,13,15\mod16$, $C_2B_2^2\equiv5,7,15\mod16$\\ 
&&2&$C_2A_2^2\equiv5,7,15\mod16$, $C_2B_2^2\equiv5,13,15\mod16$\\
 \hline
 
0&1&0&$C_2B_2^2\equiv5,13,15\mod16$, $C_2A_2^2\equiv5,7,15\mod16$\\ 
&&2&$C_2B_2^2\equiv5,7,15\mod16$, $C_2A_2^2\equiv5,13,15\mod16$\\
 \hline

\end{tabular}
\end{center}
\caption{\label{xconstantplus2} Cases where $\omega_2(t)=+1$ for every fiber at $t\in\Q$ of the surface $\mathscr{F}_{A,B,C}$. 
}
\end{table}
\end{center}

\begin{center}
\begin{table}[!htb]
\begin{center}
\begin{tabular}{| c | c | c | c |}
\hline
$v_2(A)[2]$&$v_2(B)[2]$&$v_2(C)[4]$ & Additional conditions  
\\

\hline
0&0&0 & $C_2A_2^2,C_2B_2^2\equiv 1,7,9,11\mod16$\\ 
&&1&$C_2\equiv3\mod8$\\
&&2&$C_2A_2^2,C_2B_2^2\equiv 5,7,9,15\mod16$\\
&&3&$C_2\equiv5\mod8$\\
 \hline

1&0&0&$C_2A_2^2\equiv5,7,9,15\mod16$, $C_2B_2^2\equiv1,7,9,11\mod16$\\ 
&&1,3&\\ 
&&2&$C_2A_2^2\equiv5,7,9,15\mod16$, $C_2B_2^2\equiv1,7,9,11\mod16$\\
 \hline
 
0&1&0&$C_2B_2^2\equiv5,7,9,15\mod16$, $C_2A_2^2\equiv1,7,9,11\mod16$\\ 
&&1,3&\\
&&2&$C_2B_2^2\equiv1,7,9,11\mod16$, $C_2A_2^2\equiv5,7,9,15\mod16$\\
 \hline

\end{tabular}
\end{center}
\caption{\label{xconstantmoins2} Cases where $\omega_2(t)=-1$ for every fiber $\mathscr{F}_t$ of the surface $\mathscr{F}_{A,B,C}$. 
}
\end{table}
\end{center}

\begin{proof}
For every choice of $m,n\in\Z$ coprime, let $\E_{m,n}:y^2=x^3+C(A^2m^4+B^2n^4)x$ be an elliptic curve $\Q$-isomorphic to $E_{\frac{m}{n}}$. We know the formula of the local root number at 2 by \cite[Lemme 4.7]{VA} (that we recall at Lemma \ref{VAx}). Moreover, recall that if $t$ is an odd integer, one has $$\left(\frac{-2}{t}\right)=\begin{cases}
+1 & \text{if }t\equiv1,3\mod8,\\
-1 & \text{otherwise.}
\end{cases}$$

Put, for every $m,n\in\Z$ coprime integers, $P(m,n)=C(A^2m^4+B^2n^4)$. 
We have
$$P(m,n)=2^{v_2(C)}C_2\left(2^{2v_2(A)+4v_2(m)}A_2^2m_2^4+2^{2v_2(B)+4v_2(n)}B_2^2n_2^4\right)$$
and thus, when $v_2(A)+2v_2(m)\not=v_2(B)+2v_2(n)$, one has
$$v_2(P(m,n))=v_2(C)+2\min(v_2(A)+2v_2(m),v_2(B)+2v_2(n)),$$ and $$P(m,n)_{2}\equiv \begin{cases} C_2(4A_2^2+B_2^2)\mod 4&\text{if }v_2(A) odd\\
C_2(A_2^2+4B_2^2)\mod 4&\text{if }v_2(B) odd\\
C_2B_2^2&\text{if }v_2(A)\geq2\text{ and even}\\
C_2A_2^2& \text{if }v_2(B)\geq2\text{and even}\end{cases}$$

When both $v_2(A)$ and $v_2(B)$ are even, it is possible that $v_2(A)+2v_2(m)=v_2(B)+2v_2(n)$, and in this case one has
$$v_2(P(m,n))=v_2(C)+2v_2(A)+4v_2(m)+2,$$ and $$P(m,n)_2\equiv C_2\left(\frac{A_2^2+B_2^2}{4}\right)\mod 16.$$


Suppose first $v_2(A)+2v_2(m)<v_2(B)+2v_2(n)$. In this case, we have $$v_2(P(m,n))=v_2(C)+2v_2(A)+4v_2(n)\equiv v_2(C)+2v_2(A)\mod4$$ and $$P(m,n)_2\equiv \begin{cases}B_2^2C_2\mod16&\text{if }v_2(A)\text{ even}\\5C_2&\text{if }v_2(A)\text{ odd and $B_2^2\equiv1\mod16$.}\\13C_2&\text{if }v_2(A)\text{ odd and $B_2^2\equiv9\mod16$.}\end{cases}$$. We have, 
\begin{equation}\label{j1728w2sup}
W_2(\E_{m,n})=\begin{cases}
-1 & \text{if }v_2(C)+2v_2(A)\equiv1,3\mod4\text{ and }P(m,n)_2\equiv1,3\mod8,\\
&\text{or if }v_2(C)+2v_2(A)\equiv0\mod4\text{ and }P(m,n)_2\equiv1,5,9,11,13,15\mod16,\\
& \text{or if }v_2(C)+2v_2(A)\equiv2\mod4\text{ and }P(m,n)_2\equiv1,3,5,7,11,15\mod16,\\
+1 & \text{otherwise.}
\end{cases}
\end{equation}
and thus
$$\omega_2(\E_t)=\begin{cases}
-1& \text{if }v_2(C)+2v_2(A)\equiv1,3\mod4\\
&\text{or if }v_2(C)+2v_2(A)\equiv2\mod4\text{ and }P(m,n)_2\equiv5,7,9,15\mod16\\
&\text{or if }v_2(C)+2v_2(A)\equiv0\mod4\text{ and }P(m,n)_2\equiv1,7,9,11\mod8\\

+1 &\text{otherwise.}
\end{cases}$$

The equation for $v_2(A)+2v_2(m)>v_2(B)+2v_2(n)$ is identical, with the role of $A$ and $B$ swapped since the equation of the surface is symmetric.

Moreover, note that as we supposed that $A$ and $B$ are coprime, at most one of them is divisible by $2$. 

Comparing formulas for $v_2(A)+2v_2(m)>v_2(B)+2v_2(n)$ and $v_2(A)+2v_2(m)<v_2(B)+2v_2(n)$, we get some of the entries in Tables \ref{xconstantplus2} and \ref{xconstantmoins2}. Observe moreover that, given that $A^2,B^2\equiv1,9\mod16$, some cases describe by the conditions of one line of each formula are not possible. The only work left is to study more into details the case where both $v_2(A)$ and $v_2(B)$ are both even.

For these exceptions, we proceed to a more raffined sorting.

According to the values of $A^2m^4,B^2n_2^4$ (among $1,9,17,25\mod32$), we find the possible values of $P(m,n)_2$. We always have in this case $v_2(P(m,n))=v_2(C)+1\mod4$. Hence
Thus we have $$\omega_2(t)=\begin{cases}
-\left(\frac{-2}{C_2}\right)&\text{if } v_2(C)\equiv0,2\mod4\text{ and }P(m,n)_2\equiv1,3\mod8\\
&\text{or if }v_2(C)\equiv3\mod4\text{ and }P(m,n)_2\equiv1,5,9,11,13,15\mod16\\
&\text{or if }v_2(C)\equiv1\mod4\text{ and }P(m,n)_2\equiv1,3,5,7,11,15\mod16\\
\left(\frac{-2}{C_2}\right) &\text{otherwise.}\end{cases}$$

Observe that $m^4$ and $n_2^4$ can take the values $1,17\mod32$. Therefore, choosing a value of $n'$ such that $n'^4\equiv17n^4\mod32$, we have $P(m,n')\equiv9P(m,n)_2\mod16$. Therefore, we have $P(m,n)_2\in\{1,9\mod16\}$ if $A^2\equiv B^2\mod16$, and $P(m,n)_2\in\{5,13\}$ if $A^2\not\equiv B^2\mod16$, .

This means that $\omega_2(t)$ is non-constant when $v_2(C)\equiv1\mod4$ and $C_2\equiv11,15\mod16$, and when $v_2(C)\equiv3\mod4$ and $C_2\equiv1,5\mod16$. 

We obtain thus that in that case $$\omega_2(t)=\begin{cases}
-1&\text{if } v_2(C)\equiv0,2\mod4\\
&\text{or if }v_2(C)\equiv3\mod4\text{ and }C_2\equiv1\mod8\\
&\text{or if }v_2(C)\equiv1\mod4\text{ and }C_2\equiv3\mod8\\
+1&\text{if }v_2(C)\equiv3\mod4\text{ and }C_2\equiv5\mod8\\
&\text{or if }v_2(C)\equiv1\mod4\text{ and }C_2\equiv7\mod8\\
\text{non-constant}&\text{otherwise.}
\end{cases}$$
Comparing with the formula when $v_2(A)+2v_2(m)\not=v_2(B)+2v_2(n)$, we complete the Tables \ref{xconstantplus2} and \ref{xconstantmoins2}. In particular, there is no cases were $w_2(t)=+1$ for all $t\in\Q$ when $v_2(A)\equiv v_2(B)\equiv0\mod2$.

\end{proof}

\bibliographystyle{alpha}
\bibliography{bibliothese}
\end{otherlanguage}

\end{document}